\documentclass[11pt]{amsart} \textwidth=14.5cm \oddsidemargin=1cm
\usepackage{amsmath,amssymb,latexsym, amsthm, amscd, mathrsfs, stmaryrd, tikz}
\usepackage[linktocpage=true]{hyperref}
\usepackage{color}
\usepackage[all]{xy}

\usepackage{verbatim}

\newtheorem{thm}{Theorem} [section]
\theoremstyle{definition}

\newtheorem{rem}[thm]{Remark}
\theoremstyle{plain}
\newtheorem{prop}[thm]{Proposition}
\newtheorem{lem}[thm]{Lemma}
\newtheorem{cor}[thm]{Corollary}

\numberwithin{equation}{section}

\newcommand{\A}{\mathcal A}
\newcommand{\al}{\alpha}

\newcommand{\C}{\mathbb C}
\newcommand{\D}{D(2|1;\ka)}

\newcommand{\ep}{\epsilon}

\newcommand{\hf}{{\Small \frac12}}

\newcommand{\g}{\mathfrak{g}}
\newcommand{\gl}{\mathfrak{gl}}
\newcommand{\h}{\mathfrak{h}}
\newcommand{\ka}{\zeta}

\newcommand{\la}{\lambda}

\newcommand{\mc}{\mathcal}
\newcommand{\mf}{\mathfrak}
\newcommand{\N}{\mathbb N}

\newcommand{\ov}{\overline}

\newcommand{\Z}{\mathbb Z}

\newcommand{\oa}{{\bar 0}}
\newcommand{\ob}{{\bar 1}}
\newcommand{\vare}{\epsilon} 

\newcommand{\ch}{{\rm ch}}

\newcommand{\pri}{{0}}

\allowdisplaybreaks

\title[Blocks and characters of $G(3)$-modules of non-integral weights] {Blocks and characters of $G(3)$-modules of non-integral weights}

\author[Chen]{Chih-Whi Chen}
\address{Department of Mathematics, National Central University, Chung-Li, Taiwan 32054} \email{cwchen@math.ncu.edu.tw}
\author[Cheng]{Shun-Jen Cheng}
\address{Institute of Mathematics, Academia Sinica, Taipei, Taiwan 10617} \email{chengsj@math.sinica.edu.tw}
\author[Luo]{Li Luo}
\address{School of mathematical Sciences, Shanghai Key Laboratory of Pure Mathematics and Mathematical Practice, East China Normal University, Shanghai 200241, China}
\email{lluo@math.ecnu.edu.cn}

\allowdisplaybreaks

\begin{document}

\begin{abstract}
	We classify blocks in the BGG category $\mc O$ of modules of non-integral weights for
the exceptional Lie superalgebra $G(3)$. We compute the characters for tilting modules   of non-integral weights in $\mc O$.  Reduction methods are established to connect non-integral blocks of $G(3)$ with blocks of the special linear Lie algebra $\mf{sl}(2)$, the exceptional Lie algebra $G_2$, the general linear Lie superalgebras $\gl(1|1)$, $\gl(2|1)$ and the ortho-symplectic Lie superalgebra $\mf{osp}(3|2)$.
\end{abstract}

\maketitle

\setcounter{tocdepth}{1}
\tableofcontents

\section{Introduction}

	\subsection{} The Bernstein-Gelfand-Gelfand categories $\mc O$ over finite-dimensional complex simple Lie superalgebras have attracted much attention in the last two decades. A central problem in this study is the determination of the irreducible characters in $\mc O$. One of the important progresses is the discovery of a remarkable connection between the representation theories of Lie algebras and Lie superalgebras. As a consequence, the irreducible characters of finite-dimensional simple Lie superalgebras of types ABCD are determined by the Kazhdan-Lusztig polynomials of their classical counterparts which are classical Lie algebras. Since the KL polynomials for these classical Lie algebra are well-known, this gives a satisfactory solution to the irreducible character problem of simple classical Lie superalgebras in the BGG categories of modules with integral highest weights \cite{CLW11,CLW15,Bao17,BLW17,BW18}. In \cite{CMW13} it was shown that in the case of type A Lie superalgebras the computation of irreducible characters of non-integral highest weight can be reduced to that of integral highest weight. As these are known \cite{CLW15} (see also \cite{BLW17}), this completely solves the irreducible character problem in $\mc O$ for type A Lie superalgebras.

In the list of finite-dimensional complex simple Lie superalgebras classified in \cite{Kac77} there are three {\em exceptional} Lie superalgebras, namely $\D$, $G(3)$ and $F(3|1)$. In principle, the irreducible characters of these Lie superalgebras in $\mc O$ can also be obtained from those of certain infinite-rank symmetrizable Kac-Moody Lie algebras via super duality, see \cite[Corollary 4.12]{CKW15}. However, in contrast to the type ABCD case, the irreducible characters of these Kac-Moody Lie algebras are not known, and so this approach to the irreducible character problem is not applicable for the exceptional Lie superalgebras.
	
	A systematic study of the irreducible character formula in the BGG categories for $\D$ and $G(3)$ was initiated by the Weiqiang Wang and the second author in \cite{CW17} and \cite{CW18}, where they gave a solution of the irreducible character problem in the case of integral highest weights. Finite-dimensional modules have been studied in \cite{Ger00, Ma14, SZ16}. Subsequently, the authors of the present article solved the irreducible character problem of $\D$-modules in the non-integral weight case in \cite{CCL20}, thus completing the computation of irreducible characters for $D(2|1,\zeta)$ in $\mc O$.

 \subsection{}	The objective of the present paper is to address irreducible character problem for the
	remaining cases not treated in \cite{CW18}, i.e., for the irreducible $G(3)$-modules of non-integral highest weights.  We solve this problem by  providing formulas for Verma flag multiplicities of all tilting modules in $\mc O$ of non-integral highest weight.

		Recall that the even subalgebra  of the  exceptional Lie superalgebras $G(3)$ is $G_2\oplus \mf{sl}(2)$, while its odd part is the exterior tensor product of the irreducible $7$-dimensional $G_2$-module and the $2$-dimensional natural $\mf{sl}(2)$-module. As is well-known, the Weyl group of a Lie superalgebra does not control the blocks of $\mc O$ in general. The blocks that are controlled by the Weyl group are {\em typical}, while those that are not are {\em atypical}.

	 To  describe these blocks, we make use of the super Jantzen sum formula applied to the contragredient Lie superalgebra  $G(3)$, see Proposition \ref{prop:flags}. Besides providing insights into the structure of Verma modules, it is also useful for the determination of simple modules in a block.

Our first result is a classification of atypical blocks, namely, an explicit description of highest weights of the simple modules in these blocks. Here, an important ingredient in our proof is the description of the central characters by Sergeev in \cite{Serg}.  This enables us to describe the linkage class in terms of integral Weyl group and isotropic atypical roots in Theorem \ref{thm::blocks}. This is compatible with the description of the linkage principle for integral weight modules over basic classical Lie superalgebras, see e.g., \cite[Lemma 2.1]{Se03}, or \cite[Lemma 12.3]{CM17}.

		\subsection{} Our strategy of computing characters of tilting modules is similar to \cite{CCL20} and is divided into two methods: {\em reduction methods} and the application of {\em translation functors}.
		
		We first classify  all atypical blocks according to their their $G_2$ integral Weyl groups. It turns out that the Serre subcategory generated by simple modules of atypical non-integral highest weights can be decomposed into the direct sum of five full subcategories (see Section \ref{subsect::classification}): \[\mc O_{e}, ~\mc O_{\Z_2},~\mc O_{\Z_2\times \Z_2},~\mc O_{S_3},~\mc O_{W_{G_2}},\]
where the subscripts above indicate the integral Weyl groups as subgroups of the Weyl group of $G_2$. This allows us to treat the character problem of atypical blocks in each subcategory above in a uniform fashion.
		
		The  method of translation functor follows that of \cite{CW17, CW18, CCL20}. Namely, for a given tilting module $T$ whose character we wish to compute, we make a reasonable choice of an initial tilting module. Then we  apply the translation functor by tensoring the initial module with the $31$-dimensional adjoint module to obtain a direct sum the tilting modules, a summand of is $T$. A detailed analysis of the character by means of Proposition \ref{prop:flags} then enables us to derive the character of $T$.

 \subsection{} When a block is typical, it is known that it is equivalent to a block of the underlying Lie algebra under the additional assumption that it is {\em strongly typical} \cite{Gor02b}. While strongly typical blocks are typical, the converse is not true for $G(3)$.  Nevertheless, we show in Theorem \ref{thm::tysty} that in the case of $G(3)$ any typical block is still equivalent to a block of the underlying Lie algebra $\mf{sl}(2)\oplus G_2$.

\subsection{}  In the case of type $A$ Lie superalgebras, Mazorchuk, Wang and the second author reduced the irreducible character problem of an arbitrary highest weight to that of integral highest weight in \cite{CMW13}. The main tools used there are twisting, odd reflection, and parabolic induction functors. Along this line of reduction, there have been some applications to other types of Lie superalgebras in \cite{Ch16, CCL20}.  In the present paper, we establish equivalences of categories connecting the blocks in $\mc O_{e}, \mc O_{\Z_2}$ with blocks  of the special linear Lie algebra $\mf{sl}(2)$, the exceptional Lie algebra $G_2$, the general linear Lie superalgebras $\gl(1|1)$, $\gl(2|1)$ and the ortho-symplectic Lie superalgebra $\mf{osp}(3|2)$ in Sections \ref{Sect::4} and \ref{Sect::1intblocks}.
This solves the irreducible character problem in these $G(3)$-blocks.

\subsection{} We conclude this introduction with an outline of the paper. In Section \ref{sect::preli} we establish notations and conventions used throughout. We study the integral Weyl groups of non-integral weights in the Section \ref{subsect::intWgp}. We derive some consequences of the super Jantzen sum formula in the Section \ref{subsect::ConJanFor}.

In Section \ref{sect::clablock}, we describe the highest weights of simple modules in atypical blocks.
We classify in Section \ref{subsect::classification} all atypical blocks according to their integral Weyl groups. 

We study in Section \ref{Sect::4} the typical blocks and blocks in $\mc O_{e}$, which we refer to as generic. We show that any typical block is always equivalent to a strongly typical block, while the generic blocks are equivalent to the principal block of the category of finite-dimensional $\gl(1|1)$-modules.

In Section \ref{Sect::1intblocks},  we investigate blocks in $\mc O_{\Z_2}$, which turn out to be equivalent to a direct sum of certain blocks of BGG category $\mc O$ for $\mf{sl}(2)$, $\gl(1|1)$, $\gl(2|1)$ and $\mf{osp}(3|2)$ as highest weight categories.

  In the remaining sections, we focus on the characters of tilting modules in blocks $\mc O_{\Z_2\times \Z_2}$, $\mc O_{S_3}$ and $\mc O_{W_{G_2}}$.
  We study in Section \ref{sect::OVcase} the characters of tilting modules of blocks in $\mc O_{\Z_2\times \Z_2}$, while in Section \ref{sect::OS3case} we study the characters of tilting modules in blocks in $\mc O_{S_3}$. Finally, the characters of tilting modules in $\mc O_{W_{G_2}}$ are computed in Section \ref{sect::OWG2case}.

In Appendix \ref{sect::app}, we give a proof for the characters of tilting modules of $\mf{osp}(3|2)$ for the convenience of the reader, as we were not able to find a reference in the literature. Note that in the case of integral weights such formulas can also be obtained from the results in \cite{BW18} in principle. Indeed, they have been computed by Bao and Wang earlier. We are grateful to them for communication.

\vskip 0.5cm

{{\bf Acknowledgments}. The first two authors are partially supported by MoST grants of the R.O.C.   The third author is partially supported by the Science and Technology Commission of Shanghai Municipality (grant No. 18dz2271000) and the NSF of China (grant No. 11871214).}
	
\section{Preliminaries} \label{sect::preli}
\subsection{Lie superalgebras $G(3)$}
\subsubsection{ Structure of $G(3)$}
The Lie superalgebra $G(3)$ is the contragredient Lie superalgebra associated with the following Cartan matrix:
\begin{align}\label{Cartan:matrix}
\begin{pmatrix}
2 & -1 & 0 \\
-3 & 2 & -1 \\
0 & 1 & 0
\end{pmatrix}.
\end{align}
We refer the reader to \cite{Kac77}, \cite[Section 5.3]{Mu12} for more details, including the parity of simple roots.

Throughout, unless otherwise stated, we denote the exceptional Lie superalgebra $G(3)$ over $\mathbb{C}$ by $\mathfrak{g}=\mathfrak{g}_{\overline{0}}\oplus \mathfrak{g}_{\overline{1}}$. The underlying even subalgebra $\mathfrak{g}_{\overline{0} }$ is isomorphic to $G_2\oplus \mathfrak{sl}(2)$, and the odd part $\mathfrak{g}_{\overline{1}}\cong \C^7\boxtimes \C^2$, where $\C^7$ and $\C^2$ are the $7$-dimensional simple $G_2$-module and the natural $\mathfrak{sl}(2)$-module, respectively.

\subsubsection{Standard root system and Weyl group}
We denote the associated Cartan subalgebra of $\mf g$ by $\mathfrak{h}$. To describe the roots of $\mf g$, we set   $$\mathfrak{h}^*=\mathbb{C}\delta\oplus\mathbb{C}\epsilon_1\oplus\mathbb{C}\epsilon_2,$$ which is the $\mathbb{C}$-linear space with basis $\{\delta,\epsilon_1,\epsilon_2\}$. Let $\epsilon_3=-\epsilon_1-\epsilon_2$ so that $\epsilon_1+\epsilon_2+\epsilon_3=0.$ The dual space  $\mathfrak{h}^*$ is equipped  with a bilinear form $(\cdot,\cdot)$ satisfying
$$(\delta,\delta)=-(\epsilon_i,\epsilon_i)=-2,\quad (\delta,\epsilon_i)=(\epsilon_i,\delta)=0,\quad (\epsilon_i,\epsilon_j)=-1, \quad\mbox{for $1\leq i\neq j\leq3$}.$$

The simple system associated with the Cartan matrix  \eqref{Cartan:matrix} is realized as $$\Pi=\{\alpha_1:=\epsilon_2-\epsilon_1, \alpha_2:=\epsilon_1,\alpha_3:=\delta+\epsilon_3\}.$$ The Dynkin diagram associated to $\Pi$ is depicted as follows:
\vskip .5cm
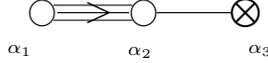
\begin{center}
	\begin{tikzpicture}
	\node at (0,0) {$\bigcirc$};
	\draw (0.2,0)--(1.155,0);
	\draw (0.15,0.1)--(1.2,0.1);
	\draw (0.15,-0.1)--(1.2,-0.1);
	\node at (1.35,0) {$\bigcirc$};
	\node at (0.75,0) {\Large $>$};
	\draw (1.52,0)--(2.52,0);
	\node at (2.7,0) {$\bigotimes$};
	\node at (-0.3,-.5) {\tiny $\al_1$};
	\node at (1.3,-.53) {\tiny $\al_2$};
	\node at (2.9,-.5) {\tiny $\al_3$};
	\end{tikzpicture}
\end{center}

The root system associated to $\Pi$ is  $\Phi=\Phi_{\overline{0}}\cup\Phi_{\overline{1}}$, where the set of positive even roots $ \Phi_{\overline{0}}^+$ and the set of positive odd roots  $\Phi_{\overline{1}}^+$ respectively are
$$\Phi_{\overline{0}}^+=\{2\delta,\epsilon_1,\epsilon_2,-\epsilon_3,\epsilon_2-\epsilon_1,
\epsilon_1-\epsilon_3,\epsilon_2-\epsilon_3\}, \qquad \Phi_{\overline{1}}^+=\{\delta,\delta\pm\epsilon_i~|~i=1,2,3\}.$$ We have the following subsets of positive roots:
\begin{align*}
  \Phi_{\overline{1},\otimes}^+&:=
\{\alpha\in\Phi_{\overline{1}}^+~|~(\alpha,\alpha)=0\}=\{\delta\pm\epsilon_i~|~i=1,2,3\},\\
\Phi_{\overline{1},\bullet}^+&:=\Phi_{\overline{1}}^+\setminus\Phi_{\overline{1},\otimes}^+=\{\delta\},\\
\Phi_{\overline{0},\circ}^+&:=
\{\alpha\in\Phi_{\overline{0}}^+~|~\frac{1}{2}\alpha\not\in\Phi_{\overline{1},\bullet}^+\}
=\{\epsilon_1,\epsilon_2,-\epsilon_3,\epsilon_2-\epsilon_1,
\epsilon_1-\epsilon_3,\epsilon_2-\epsilon_3\}.
\end{align*}
Observe that $\{\alpha_1,\alpha_2\}$ forms a simple system for the simple Lie algebra $G_2$. Let $$\omega_1 =\vare_1+2\vare_2,\qquad\omega_2 = \vare_1+\vare_2$$ denote the corresponding fundamental weights of $G_2$.

 The Weyl vector for $\mathfrak{g}$ is $$\rho:= \rho_\oa -\rho_\ob=-\frac{5}{2}\delta+2\epsilon_1+3\epsilon_2,$$ where
$\rho_\oa:=\frac{1}{2}\sum_{\alpha\in\Phi_{\overline{0}}^+}\alpha=\delta+2\epsilon_1+3\epsilon_2$ and  $\rho_\ob:=
\frac{1}{2}\sum_{\beta\in\Phi_{\overline{1}}^+}\beta= \frac{7}{2}\delta$.

For any $\alpha\in\Phi_{\overline{0}}$, we denote by $\alpha^{\vee}\in\mathfrak{h}$ the corresponding coroot determined by
$$\langle\lambda,\alpha^{\vee}\rangle=\frac{2(\lambda,\alpha)}{(\alpha,\alpha)}, \quad \text{ for } \lambda\in\mathfrak{h}^\ast.$$ The associated reflection $s_\alpha:\h^\ast\rightarrow\h^\ast$ is defined by $s_\alpha \la =\la -\langle \la, \alpha^\vee \rangle \alpha. $
We set $$s_0=s_{2\delta},\quad s_1 =s_{\alpha_1},\quad s_2 =s_{\alpha_2}.$$  The Weyl group of $\mf g$ is $W=\langle s_{\alpha}|~\alpha \in \Phi_\oa^+\rangle = S_2 \times W_{G_2}$, where $S_2$ and $W_{G_2}$ are the Weyl groups of $\mf{sl}(2)$ and $G_2$, respectively. Namely, $S_2= \langle s_0 \rangle$ and $W_{G_2} = \langle s_1, s_2\rangle$.

\subsubsection{Other simple systems} \label{subsect::simplesys}
There are four simple systems for $\mf g$, up to $W$-conjugacy.
\begin{align}
&\Pi^0 := \Pi = \{\vare_2-\vare_1,~\vare_1,~\delta+\vare_3\},\\
&\Pi^1 = \{ \vare_2-\vare_1,~\delta -\vare_2,~-\delta-\vare_3\},  \\
&\Pi^2 = \{\delta-\vare_1,~-\delta+\vare_2,~\vare_1\},\\
&\Pi^3 =  \{\vare_2-\vare_1,~\vare_1-\delta,~\delta\}.
\end{align}
The latter three simple systems can be obtained from the standard simple system $\Pi$ by applying   odd reflections associated with isotropic simple roots (see, e.g., \cite{PS94}).
Let us denote the corresponding four Borel subalgebras by $\mf b^0:=\mf b$, $\mf b^1$,  $\mf b^2$, and  $\mf b^3$, respectively.

\vskip 1cm

\subsection{Atypical non-integral weights}
\subsubsection{The symbol notation}
We recall  the  correspondence between the following symbols and weights in $\h^\ast$ introduced in \cite[Section 3.2]{CW18}.  In the present paper,  we identify weights with symbols as follows
\begin{align} \label{eq::wtid}
&\la \equiv \left[d~ \middle|~    {b}/{2},  (3a+b)/{2}, - (3a+2b)/{2}\right], \text{~for  } \la = d\delta +a\omega_1 +b\omega_2.
\end{align}
We note that $$[d|x,y,z] \equiv d\delta +\frac{2}{3}(y-x)\omega_1 + 2x\omega_2 ,$$ for any complex numbers $x,y,z$ with $x+y+z =0$.

For $\la=d\delta+a\omega_1+b\omega_2$ we have:
\begin{equation}\label{eq::Weylgpcoroot}
\left\{
\begin{array}{ccc}
\langle\lambda,(2\delta)^\vee\rangle=d; & & \\
\langle \la ,\vare_1^\vee\rangle =b, & \langle \la ,\vare_2^\vee\rangle =3a+b, & \langle \la ,(-\vare_3)^\vee\rangle =3a+2b, \\
\langle \la ,(\vare_2-\vare_1)^\vee\rangle =a,& \langle \la ,(\vare_1-\vare_3)^\vee\rangle =a+b,& \langle \la ,(\vare_2-\vare_3)^\vee\rangle =2a+b.
\end{array}
\right.
\end{equation}
 By \cite[Proposition 3.2]{CW18}, the elements $s_i$ ($0\leq i\leq 2$) act as follows:
\begin{align}
& s_0[d|x,y,z] = [-d|x,y,z],\quad s_1[d|x,y,z]= [d|y,x,z], \quad s_2[d|x,y,z]= [d|-x,-z,-y].
\end{align}  Also, for $w\in W_{G_2}$ we define the following notation
\begin{align*}
&w[d|~x,y,z] = [d|~w(x,y,z)].
\end{align*}

\subsubsection{ Atypicality}

\vskip 0.5cm
Let $\gamma \in \Phi_\oa$. A weight $\la \in \h^\ast$ is called $\gamma$-integral  if $\langle \la-\rho ,\gamma^{\vee} \rangle \in \Z$, i.e.,  $$\langle \la, \gamma^\vee \rangle \in \begin{cases} \frac{1}{2}+\Z, &\mbox{for $\gamma =\pm 2\delta$;}\\
	\Z, &\mbox{for $\gamma \not = \pm 2\delta$.}
\end{cases}$$ The weight $\la \in \h^\ast$ is called {\em non-integral} if $\la$ is not $\gamma$-integral with respect to some $\gamma \in \Phi_{\oa}^+$. We define the associated {\em integral root system} $$\Phi_{[\la]}:=\{\gamma\in \Phi_\oa|~ \la \text{ is }\gamma\text{-integral} \}$$
and the {\em integral Weyl group}
$$W_\la:=\langle s_\gamma| ~\gamma \in \Phi_{[\la]} \rangle.$$
\vskip 0.5cm

For any $\la \in \h^\ast$, we define   $$Z(\la) := \{\gamma\in \Phi_{\oa,\circ}^+|~\langle\la,\gamma^\vee\rangle\in \Z\},$$
$$A(\la):= \{\alpha\in \Phi_{\ob,\otimes}^+|~(\la ,\alpha) =0\}.$$ The weight $\la$ is called {\em typical} (respectively {\em atypical}) if $A(\la)$ is empty (respectively non-empty).

\vskip 0.5cm

\subsubsection{Integral Weyl group} \label{subsect::intWgp}
We observe that $\Phi_{[\la]}$ is identical to the integral root system defined in \cite[Section 3.4]{Hu08} when regarded as $\mf g_\oa$-weights since $\rho_\ob =\frac{7}{2}\delta$.  For each $w\in W$, set $$\Lambda^w:=\begin{cases} \delta +2\Z\delta +\Z\Phi_{\oa,\circ}^+, &\mbox{for  $w\notin W_{G_2}$;}\\
\Z\delta +Z\Phi_{\oa,\circ}^+, &\mbox{for $w\in W_{G_2}$,}
\end{cases}$$ where $\Z\Phi_{\oa,\circ}^+ = \Z\vare_1 +\Z\vare_2$ by definition.  The following is an analogue of \cite[Theorem 3.4]{Hu08}.

\begin{prop} \label{prop::IntWeylGp}
	Let $\la \in \h^\ast$. Then we have
	$W_\la =  \{w\in W|~w\la-\la \in \Lambda^w\}.$
\end{prop}
\begin{proof} Let $W_{[\la]}:=\{w\in W|~w\la-\la\in \Lambda^w\}.$ 	We first show that $W_\la \subseteq W_{[\la]}$. To see this, let $s_\gamma\in W_{\la}$ with $\gamma \in \Phi_\oa$.
	
	If $\gamma \in \Phi_{_\oa}  \backslash\{\pm 2\delta\}$ with $\langle \la - \rho, \gamma^\vee \rangle \in \Z$, then  $\langle\la, \gamma^\vee \rangle \in \Z$, which implies that $s_\gamma \in W_{[\la]}.$  If $\gamma =\pm 2\delta$ then $\langle \la, \gamma \rangle \in \frac{1}{2}+\Z$. Therefore we have   $s_{\gamma} \la -\la = -\langle \la, \gamma^\vee \rangle 2\delta \in \delta +2\Z\delta\subseteq \Lambda^w$, which implies that $s_{\gamma}\in W_{[\la]}$.
	
	Let $w \in W_{[\la]}$. If $w\in W_{G_2}$ then by \cite[Theorem 3.4]{Hu08} it follows that  $w\in W_\la$. Suppose that $w = s_0w'$ for some $w'\in W_{G_2}$. Then \begin{align}
	&w\la -\la \in \Lambda^w \Leftrightarrow (s_0w'\la -w'\la) + (w'\la -\la) \in \Lambda^w,
	\end{align} which implies that $w'\la -\la \in \Z\Phi^+_{\oa,\circ}$ and $s_0w'\la -w'\la \in \delta +2\Z\delta$. In particular, we have $w'\in W_\la$ and $$2\langle \la,(2\delta)^\vee \rangle\delta= 2\langle w'\la, (2\delta)^\vee \rangle \delta \in \delta +2\Z\delta.$$ This implies that $s_{0} \in W_\la$. We may conclude that $w =s_{0} w' \in W_\la$.  It follows that $W_{[\la]}\subseteq W_\la$. This completes the proof.
\end{proof}

In the sequel, for a given weight $\la \in \h^\ast$, we define $W_\la^0\subseteq \langle s_{0}\rangle$ and $W_\la^1 \subseteq W_{G_2}$ to be subgroups of $W_\la$ so that
\begin{align} \label{eq::210}
W_\la =W_\la^0\times W_\la^1.
\end{align}

\vskip 0.5cm

\subsection{BGG category $\mc O$}
\subsubsection{}

The set of positive roots $\Phi^+$ gives rise to a triangular decomposition $\mf g=\mf n^- \oplus \mf h \oplus \mf n^+$ with corresponding Borel subalgebra $\mf b  =\mf h \oplus \mf n^+$ . Denote by $\mc O$ the BGG category of $\mf g$-modules that are finitely generated over $U(\mf g)$, semisimple as $\mf h$-, and locally finite as $\mf b$-modules.

In the present paper, we denote by $L_\la$, $M_\la$ and $T_{\la}$ the irreducible, Verma and tilting modules of  highest weight $\la -\rho$, respectively. Let $P_\la$ be the projective cover of $L_\la$ in $\mc O$.

For a given module $M\in \mc O$, we denote by $[M:L_\la]$ the multiplicities of $L_\la$ in a Jordan-H\"older series of $M$. A module $M\in \mc O$ is said to have a Verma flag if for some $k\in \N$ it has a filtration by submodules
\[0 = M_0 \subseteq M_1 \subseteq \cdots \subseteq M_k = M,\]
where $M_i/M_{i-1}$ is a Verma module for $1 \leq i \leq k$. The multiplicity of $M_\la$ appearing in such a flag of $M$ is denoted by $(M:M_\la)$. It is well-known that both projective and tilting modules in $\mc O$ have Verma flags.

We have the usual BGG reciprocity and Soergel duality (cf. \cite{Soe98, Br04, CCC19}):
\begin{align}
&(T_{-\la}: M_{-\mu}) = (P_\la:M_\mu) =[M_{\mu}:L_\la],
\end{align} for $\la,\mu \in \h^\ast.$

 \subsubsection{} \label{subsect::24} It is known that $\mc O$ is a highest weight category with standard objects $M_\la$ in the sense of \cite{BS18}, see  \cite[Section 3]{CCC19} for a complete treatment. We define $\mc O^{(i)}$ to be the BGG category corresponding to the simple system $\Pi^i$ given in Section \ref{subsect::simplesys}, for $1\leq i\leq 3$. Note that $\mc O^{(i)}$ is equivalent to $\mc O$ via the identity functor as abelian categories.

Let $S\subseteq \Pi^i$. Then $S$ gives rise to a simple system and a Borel subalgebra $\mf b^{\mf l_S}$ of the corresponding Levi subalgebra $\mf l_S$ (see Section \ref{sect::51} for examples). We let $\mc O^{\mf l_S}$ denote the corresponding BGG category of $\mf l_S$-modules. For a given weight $\la \in \h^\ast$, we define $\mc O_\la^{\mf l_S}$ to be the block of $\mc O^{\mf l_S}$ containing the irreducible $\mf l_S$-module of $\mf b^{\mf l_S}$-highest weight $\la -\rho$. In particular, we denote by $\mc O_\la$ the (indecomposable) block of $\mc O$ containing $L_\la$. We let $\text{Irr}\mc O_{\la}:=\{\mu\in\h^*|L_\mu\in\mc O_\la\}$.

In the present paper, two highest weight categories $(\mc A,\leq_{\mc A})$ and $(\mc B, \leq_{\mc B})$ are called equivalent as height weight categories if there is an equivalence of $\mc A$ and $\mc B$ preserving standard objects and their orderings $\leq_{\mc A}$ and $\leq_{\mc B}$.

 \subsubsection{} \label{subsect::ConJanFor} Recall that for any $\la\in \h^\ast$, we have the following super analogue of Jantzen sum formula \cite{Gor04,Mu12}.
 \begin{prop} \label{prop::JSF}
 	\label{prop:JSF}
  The Verma module $M_\la$ has a finite filtration of submodules,
 	\begin{align*}
 	M_\la\supset M^1_\la\supset M^2_\la\supset \cdots,
 \end{align*}
 	such that
 	\begin{align}
 	\label{jantzen:sum}
 	\begin{split}
 	\sum_{i>0} &{\rm ch} M^i_\la
 	=  \sum_{\alpha\in \Phi^+_{\bar 0,\tikz\draw (0,0) circle (.3ex);},\langle\la,\alpha^\vee\rangle\in\N\setminus\{0\}} {\rm ch}M_{s_\alpha\la}
 	+ \sum_{\langle\la,(2\delta)^\vee\rangle\in\hf+\N} {\rm ch}M_{s_{2\delta}\la}
 	+ \sum_{\gamma\in \Phi^+_{\bar 1,\otimes},(\la,\gamma)=0} \frac{{\rm ch}M_{\la-\gamma}}{1+e^{-\gamma}}.
 	\end{split}
 	\end{align}
 \end{prop}

The following proposition, which is essentially \cite[Proposition 2.2]{CW18} and which was proved by applying the above super analogue of Jantzen sum formula, is used in a crucial way in the computation of  characters of tilting modules in the sequel.
\begin{prop} \label{prop:flags}
	Let $\la\in \h^\ast$, with $\alpha,\alpha_i\in\Phi^+_{\bar 0}$ ($1\le i\le k$), and $\beta,\gamma\in\Phi^+_{\bar 1,\otimes}$. Let $w=s_{\alpha_k}\cdots s_{\alpha_2} s_{\alpha_1}\in W$.
	\begin{itemize}
		\item[(1)]
		Suppose that $\langle\la,\alpha^\vee\rangle\in \N\setminus\{0\}$, for $\alpha\in\Phi^+_{\bar 0,\tikz\draw (0,0) circle (.3ex);}$, or $\langle\la,\alpha^\vee\rangle\in \hf+\N$, for $\alpha=2\delta$. Then $(T_\la: M_{s_{\alpha} \la})>0$.
		
		\item[(2)]
		Suppose that
		$\langle s_{\alpha_{i-1}}\cdots s_{\alpha_{1}}\la,\alpha_i^\vee\rangle\in \N\setminus\{0\}$, for $\alpha_i\in\Phi^+_{\bar 0,\tikz\draw (0,0) circle (.3ex);}$, or $\langle s_{\alpha_{i-1}}\cdots s_{\alpha_{1}}\la,\alpha_i^\vee\rangle\in \hf+\N$, for $\alpha_i=2\delta$,
		for all $i=1,\ldots,k$. Then $(T_\la:M_{w\la})>0$.
		
		\item[(3)]
		Suppose that $(\la,\beta)=0$. Then $(T_\la:M_{\la-\beta})>0$.
		
		\item[(4)]
		Suppose that $(\la,\beta)=0$ and furthermore $\langle s_{\alpha_{i-1}}\cdots s_{\alpha_{1}}(\la-\beta),\alpha_i^\vee\rangle \in \N\setminus\{0\}$, for $\alpha_i\in\Phi^+_{\bar 0,\tikz\draw (0,0) circle (.3ex);}$, or $\langle s_{\alpha_{i-1}}\cdots s_{\alpha_{1}}(\la-\beta),\alpha_i^\vee\rangle \in \hf+\N$, for $\alpha_i=2\delta$,
		for all $i=1,\ldots,k$.
		Then $(T_\la:M_{w(\la-\beta)})>0$.
		
		\item[(5)]
		Suppose that $(\la,\beta)=0$. If $(\la-\beta,\gamma)=0$
		and $\beta\not\ge\gamma$, then $(T_\la:M_{\la-\beta-\gamma})>0$.
		
		\item[(6)]
		Suppose that $(\la,\beta)=0$. If $(\la-\beta,\gamma)=0$ with
		 $\beta\not\ge\gamma$ and $\langle s_{\alpha_{i-1}}\cdots s_{\alpha_{1}}(\la-\beta-\gamma),\alpha_i^\vee\rangle \in \N\setminus\{0\}$, for $\alpha_i\in\Phi^+_{\bar 0,\tikz\draw (0,0) circle (.3ex);}$ or $\langle s_{\alpha_{i-1}}\cdots s_{\alpha_{1}}(\la-\beta-\gamma),\alpha_i^\vee\rangle \in \hf+\N$, for $\alpha_i=2\delta$, for all $i=1,\ldots,k$,
		then $(T_\la:M_{w(\la-\beta-\gamma)})>0$.
\item[(7)] Suppose $\langle \la, \alpha^{\vee}\rangle\in \N\setminus\{0\}$, for $\alpha\in\Phi^+_{\bar 0,\tikz\draw (0,0) circle (.3ex);}$, or $\langle \la, \alpha^{\vee}\rangle\in \hf+\N$, for $\alpha=2\delta$. If $(s_\alpha(\la),\gamma)=0$ and $\langle \la, \alpha^{\vee}\rangle\alpha\not\ge\gamma$, then $(T_\la:M_{s_\alpha(\la)-\gamma})>0$.
\item[(8)] Suppose $\langle \la, \alpha^{\vee}\rangle\in \N\setminus\{0\}$, for $\alpha\in\Phi^+_{\bar 0,\tikz\draw (0,0) circle (.3ex);}$, or $\langle \la, \alpha^{\vee}\rangle\in \hf+\N$, for $\alpha=2\delta$. Assume further that $(s_\alpha(\la),\gamma)=0$, $\langle \la, \alpha^{\vee}\rangle\alpha\not\ge\gamma$ and $\langle s_{\alpha_{i-1}}\cdots s_{\alpha_{1}}(s_\alpha(\la)-\gamma),\alpha_i^\vee\rangle \in\N\setminus\{0\}$, for $\alpha_i\in\Phi^+_{\bar 0,\tikz\draw (0,0) circle (.3ex);}$, or $\langle s_{\alpha_{i-1}}\cdots s_{\alpha_{1}}(s_\alpha(\la)-\gamma),\alpha_i^\vee\rangle\in \hf+\N$, for $\alpha_i=2\delta$, for all $i=1,\ldots,k$, then we have $(T_\la:M_{w(s_\alpha(\la)-\gamma)})>0$.	
\end{itemize}
\end{prop}

\begin{proof}
Parts (1)--(6) are proved in \cite[Proposition 2.2(1)--(6)]{CW18}. Parts (7) and (8) are variants of Parts (5) and (6), respectively, and are proved analogously.
\end{proof}

\subsubsection{}\label{sec:comp:factor}

In this Section \ref{sec:comp:factor} we let $\g=\g_{\oa}\oplus \g_{\ob}$ be a finite-dimensional simple Lie superalgebra. Suppose that $\g_\oa=\mf a\oplus\mf c$ is a direct sum of reductive Lie algebras so that $W=W^{\mf a}\times W^{\mf c}$. Let $\h$, $\h^{\mf a}$ and $\h^{\mf c}$ be, respectively, Cartan subalgebras of $\g$, $\mf a$ and $\mf c$ with $\h=\h^{\mf a}\oplus\h^{\mf c}$. We shall continue to use superscripts to denote the corresponding elements in $\h$ and $\h^*$ with respect to this decomposition. For $\mu\in{\h^\mf c}^*$ we denote by $M^{\mf c}_\mu$ and $L^{\mf c}_\mu$ the Verma and irreducible $\mf c$-modules of highest weight $\mu-\rho^{\mf c}$, respectively.

\begin{prop}\label{prop:comp:factor}
Suppose $\Pi$ is a simple system of $\g$ containing a simple system for $\mf c$ so that we have compatible triangular decompositions for $\g$ and $\mf c$. Then, for $\la\in\h^*$ and $\tau\in W^{\mf c}$ we have
\begin{align*}
[M_\la:L_{\tau\la}]=[M^{\mf c}_{\la^{\mf c}}:L^{\mf c}_{\tau\la^{\mf c}}].
\end{align*}
\end{prop}

\begin{proof}
Denote by $v_{\la-\rho}$ and $v^{\mf c}_{\la^\mf c-\rho^c}$ some highest weight vectors in $M_\la$ and $M^{\mf c}_{\la^\mf c}$, respectively.

The proposition follows from the following claim:
Suppose that $v_{\tau(\la+\rho)-\rho}$ is a nonzero primitive vector in $M_\la$ of weight $\tau(\la+\rho)-\rho$ so that
\begin{align}\label{vec:prim}
v_{\tau(\la+\rho)-\rho}=u v_{\la-\rho}, \quad u\in U({\mf n^{\mf c}}^-).
\end{align}
Then the vector
$u v^{\mf c}_{\la^\mf c-\rho^\mf c}\in M^{\mf c}_{\la^\mf c}$ is a nonzero primitive vector in $M^{\mf c}_{\la^\mf c}$ of weight $\tau(\la^\mf c+\rho^\mf c)-\rho^\mf c$. On the other hand, if $u v^{\mf c}_{\la^\mf c-\rho^\mf c}$, for $u\in U(\mf n^\mf c_-)$, is a nonzero primitive vector in $M^{\mf c}_{\la^\mf c}$ of weight $\tau(\la^\mf c+\rho^\mf c)-\rho^\mf c$, then it
determines a nonzero primitive vector in $M_\la$ of weight $\tau(\la+\rho)-\rho$ by the formula \eqref{vec:prim}.

The fact that primitive vectors in $M^\mf c_{\la^c}$ give rise to primitive vectors in $M_\la$ is clear. Below, we shall give an argument for the other direction. The claim is easily seen to hold for singular vectors. Now, suppose that we have a primitive vector $w$ of weight $\tau(\la+\rho)-\rho$ such that $U(\mf n^+)w$ lies in the span of modules generated by singular vectors. Then, by weight consideration we can assume that these singular vectors all have weights $\sigma(\la+\rho)-\rho$, $\sigma\in W^\mf c$, so that they give rise to corresponding singular vectors in $M^\mf c_{\la^c}$ as well. From this, it follows that $w$ is a primitive vector in $M^\mf c_{\la^c}$ as well. Next, suppose that $w'$ is a primitive vector such that  $U(\mf n^+)w'$ lies in the span of modules generated by singular vectors and primitive vectors as $w$ above. Similarly as before, we see that $w'$ is a primitive vector in $M^\mf c_{\la^c}$. Continuing this way, we establish the claim.
\end{proof}

\vskip 0.5cm

\subsubsection{} \label{subsubsect::Cas}

 Let $\la =[d|x,y,z]$ be a weight in the symbol notation. We compute the eigenvalue $c_\la$ of the Casimir operator on the Verma module $M_\la$ and get
 \begin{align}
 &c_\la  = (\la, \la) -(\rho ,\rho) = -2d^2 +\frac{4}{3}x^2 +\frac{4}{3}y^2 +\frac{4}{3}z^2 -(\rho ,\rho).
 \end{align} Thus, there exists a   element $\Omega$ in the center of $U(\mf g)$   that  acts on $M_\la$ by the scalar $3d^2 -2x^2-2y^2 -2z^2$, for   $\la = [d|~x,y,z].$


\section{Classification of blocks} \label{sect::clablock}
We describe atypical blocks in this section.
\subsection{Block decompositions}

\begin{lem} \label{lem::11} Let $\la\in \h^\ast$ with  $\alpha, \beta \in A(\la)$ and $\alpha\not=\beta$. Then there exists $\sigma=s_{\gamma} \in W$ with $\gamma \in \Phi_\oa^+$ such that $\sigma \alpha =\beta$ and $\sigma \la =\la$.
\end{lem}
\begin{proof} We shall prove this case by case.  \\
	(1). Suppose that $c=\pm1$ and $\alpha =\delta +c\vare_i$ and $\beta =\delta +c\vare_j$ with $i\neq j$. It follows that $(\la, \vare_i-\vare_j)=0$. We  set $\sigma =s_{\vare_i -\vare_j}$ as desired.
	\\
	(2). Suppose that $\alpha =\delta+\vare_i$ and $\beta =\delta -\vare_j$, for $i\neq j$. Let $1\leq k \leq 3 $ such that $k\neq i,j$. Then $(\la,\vare_k)=-(\la,\vare_i+\vare_j)=0$.  Set  $\sigma =s_{\vare_k}$, as desired. \\
	(3). Suppose that $\alpha =\delta +\vare_i$ and $\beta =\delta-\vare_i$. Then $(\la,\vare_i) =0$. In this case we set $\sigma =s_{\vare_i}$, as desired.
\end{proof}

\begin{cor}\label{cor::2}
	 Let $\la\in \h^\ast$ and $\alpha, \beta \in A(\la)$. Then $$\{W(\la+k\alpha)|~k\in \C\} = \{W(\la+k\beta)|~k\in \C\}.$$
\end{cor}
\begin{proof}
	By Lemma \ref{lem::11} we have  $$W(\la+k\alpha)=W\sigma(\la+k\alpha) = W(\la+\beta),$$ for any $k\in \C.$
\end{proof}

We define a relation $\sim$ on $\h^\ast$ by   $$\la \sim \mu, \text{ if }\mu =w(\la +k\alpha), \text{~ for some {$w\in W$,}} ~\alpha \in A(\la), ~k\in \C.$$  The following lemma shows that $\sim$ is an equivalence relation.

\begin{lem} \label{lem::eqiv22}
	The relation $\sim$ is an equivalence relation on $\h^\ast$.
\end{lem}
\begin{proof}
	Suppose that $\la \sim \mu$. Then $\mu =w(\la +k\alpha)$, for some $w\in W,~\alpha\in A(\la)$ and $k\in \C$ by definition. It follows that $\la =w^{-1}(\mu -kw\alpha)$ with $w\alpha \in A(\mu)$. Thus, $\mu \sim \la$.
	
   Suppose that $\la \sim \mu$ with $\mu =w(\la +k\alpha)$ as above, and in addition we suppose that $\mu \sim \nu$. By definition we have  $\nu =\sigma(\mu +\ell \beta)$, for some $\sigma \in W,~\ell\in \C$ and $\beta \in A(\mu)$. Since $w\alpha, \beta \in A(\mu)$, there exists $\tau \in W$ such that $\tau \mu =\mu$ and $\beta =\tau w \alpha$ by Lemma \ref{lem::11}. We have
   \begin{align*}
   &\nu =\sigma(\mu +\ell \beta) =\sigma(\tau \mu +\ell \tau w\alpha)= \sigma\tau(\mu+\ell w\alpha) = \sigma\tau(w(\la+k\alpha) +\ell w\alpha) =\sigma\tau w(\la+(k+\ell)\alpha),
   \end{align*} which implies that $\la \sim \nu$.
\end{proof}

Let $\chi_\nu:Z(\mathfrak{g})\rightarrow\mathbb{C}$ denote the central  character determined by $\nu \in \h^\ast$.
\begin{thm} \label{thm::3}
	Suppose that $\la,\mu \in \h^\ast$ are atypical weights. Then $\la \sim \mu$ if and only if $\chi_{\la -\rho} =\chi_{\mu-\rho}.$
\end{thm}
\begin{proof}
It follow, e.g., from \cite[(0.6.7)]{Serg} that $\la\sim\mu$ implies that $\chi_{\la-\rho}=\chi_{\mu-\rho}$.

So it remains to show that $\chi_{\la -\rho} =\chi_{\mu -\rho}$ implies that $\la \sim \mu$.  Suppose that $\la$ is an atypical weight. We write $\la$ in the symbol notation $\la =[d|a,b,c]$ with  $a+b+c = 0$. Applying the action of elements in $W$ if necessary, we may assume that $a =d$ since $\la$ is atypical. Set $k=b+\frac{d}{2}$,  then we have
	$\la =[d|d,k-\frac{d}{2}, -(k+\frac{d}{2})]$.
	
	Recall the central element $\Omega$ constructed in Section \ref{subsubsect::Cas}. We compute the eigenvalue of the element $\Omega$ on $M_\la$ as follows: \begin{align*}
	&  3d^2 -2d^2-2(k-\frac{d}{2})^2 -2(k+\frac{d}{2})^2\\
	&=3d^2 -2d^2 -2k^2 +2kd -\frac{d^2}{2} -2k^2 -2kd -\frac{d^2}{2} \\
	&=-4k^2.
	\end{align*}
	Similarly, we can write $\mu$ in the symbol form $\mu =[e|e,\ell -\frac{e}{2}, -(\ell +\frac{e}{2})]$ and infer that $c_\mu =-4\ell^2$. Since $\chi_{\la -\rho} =\chi_{\mu-\rho}$, we may conclude that $\ell =\pm k$.
	
	If $\ell = k$ then we see that $\mu = \la +(e-d)(\delta +\vare_1)$. If $\ell =-k$ then it follows from $[0|0,k,-k] = s_{\vare_2 -\vare_3} [0|0,-k,k]$ that
	\begin{align*}
	&\la \sim [0|0,k,-k]\sim  [0|0,-k,k] \sim \mu.
	\end{align*} This proves that $\la \sim \mu$ by   transitivity of $\sim$.
\end{proof}

\begin{thm} \label{thm::blocks}
	Let $\la \in \h^\ast$ with $\alpha \in A(\la)$. Then we have \begin{align}
	&\emph{Irr}\mc O_\la =\{W_\la(\la+k\alpha)|~k\in \Z\}.
	\end{align}
\end{thm}
\begin{proof} We first show that  $\text{Irr}\mc O_\la \subseteq \{W_\la(\la+k\alpha)|~k\in \Z\}.$ To see this, suppose that $\la$ and $\mu$ are two weights such that $\chi_{\la -\rho} =\chi_{\mu-\rho}$ and $\la -\mu \in \Z\Phi$. By Theorem \ref{thm::3} there exist $w\in W$, $\ell \in \C$ and $\alpha  =\delta\pm \vare_i \in A(\la)$ such that $\mu =w(\la+ \ell \alpha)$.

Suppose that $w\in W_{G_2}$. Then $\ell = \langle w\la -\mu , (2\delta)^\vee\rangle =\langle \la -\mu , (2\delta)^\vee\rangle\in \Z$. Also, we have $$\la -w\la =\la -\mu +\ell w\alpha  \in \Z\Phi.$$ By Proposition \ref{prop::IntWeylGp}, it follows that $w\in W_\la$. Therefore $\mu \in W_\la(\la +\Z\alpha)$.

Suppose that $w\notin W_{G_2}$. Then   $w=s_{2\delta}w'$, for some $w'\in W_{G_2}$. First, we assume that $\alpha =\delta+\vare_1$. There exist  $x,d\in \C$ such that $\la =[d|~d,x,-x-d]$ and \begin{align*}
&\mu  = [-d-\ell|~w'(d+\ell,x-{\ell}/{2},-x-d-{\ell}/{2})].
\end{align*}
Define $k:=2d+\ell $.  Note that $\la -\mu \in \Z\Phi$ implies that $k\in \Z$. We have
\begin{align*}
&\mu =[d-k|~w'(k-d, x+d-{k}/{2},-{k}/{2} -x)] \\
&=[d|~w'(-d,x+d,-x)] -k[1|~w'(-1,{1}/{2}, {1}/{2})] \\
&=w's_2(\la -k\alpha),
\end{align*} which reduces to the previous case. Therefore   $\mu \in W_\la(\la +\Z\alpha)$.  The case when  $\alpha =\delta -\vare_1$ can by proved by similar arguments. If follows that we have $\mu =w's_2(\la -k\alpha)$.

The remaining cases can be completed by similar arguments. In the case when $\alpha =\delta \pm \vare_2$ we have $$\mu = w's_1s_2s_1(\la -k\alpha),$$ for some $k\in \Z$. In the case when $\alpha =\delta\pm \vare_3$ we have $$\mu =w's_1(s_2s_1)^3(\la -k\alpha),$$
for some $k\in \Z$. This proves that $\text{Irr}\mc O_\la \subseteq \{W_\la(\la+k\alpha)|~k\in \Z\}.$

Finally, it follows  from  Proposition \ref{prop::JSF} that  $\text{Irr}\mc O_\la \supseteq \{W_\la(\la+k\alpha)|~k\in \Z\}$.
\end{proof}

\subsection{Description of non-integral blocks} \label{subsect::classification}

\subsubsection{} Let $\la\in \h^\ast$ be an atypical non-integral weight. We recall the subgroup $W_\la^1$ of $W_\la$ from \eqref{eq::210}. We describe $W_\la^1$ in the following theorem.
\begin{thm} \label{thm::desnonblocks}
	 Let $\la \in \h^\ast$ be an atypical and non-integral weight. Then the subgroup $W_\la^1$ is isomorphic to one of the following subgroups $$ \{e\}, \Z_2, \Z_2\times \Z_2, S_3, D_{12},$$
where $S_3$ is the symmetric group in $3$ letters and $D_{12}$ is the dihedral group of order $12$.
 Furthermore, we have
	\begin{itemize}
		\item[(1)] $W_\la^1= \{e\}$ $\Leftrightarrow$ $W_\la = \{e\}$.
		\item[(2)] $W_\la^1\cong \Z_2$ $\Leftrightarrow$ $W_\la^1 = \langle s_{\vare_2 -\vare_1} \rangle,~\langle s_{\vare_1} \rangle,~\langle  s_{\vare_1 -\vare_3}\rangle,~ \langle s_{-\vare_3}\rangle,~\langle s_{\vare_2 -\vare_3}\rangle$ or $\langle s_{\vare_2}\rangle$.
		\item[(3)] $W_\la^1\cong \Z_2\times \Z_2$ $\Leftrightarrow$ $W_\la^1 = \langle s_{-\vare_3}, s_{\vare_2-\vare_1}\rangle, ~\langle s_{\vare_1}, s_{\vare_2-\vare_3}\rangle$ or $\langle s_{\vare_2}, s_{\vare_1 -\vare_3} \rangle$.
		\item[(4)] $W_\la^1\cong S_3$ $\Leftrightarrow$ $W_\la^1 = \langle s_{\vare_1}, s_{\vare_2}, s_{-\vare_3}\rangle$.
			\item[(5)] $W_\la^1= W_{G_2}\cong D_{12}\Leftrightarrow W_\la = W_{G_2}.$
	\end{itemize}

\end{thm}
\begin{proof}
 We first note that the part $(5)$ follows from the definition that $\la$ is non-integral.

 We next show the part $(1)$. To see this, we suppose on the contrary  that $W_\la^1=\{e\}$ and $W_\la =\langle s_{2\delta} \rangle$.  We first assume that $(\la ,\delta\pm\vare_1 ) =0$. Then we have $\frac{1}{2}\langle \la ,\vare_1^\vee \rangle  = \pm \langle\la,(2\delta)^\vee \rangle$ by \eqref{eq::Weylgpcoroot}, which implies that  $\langle\la,(2\delta)^\vee \rangle  \notin \frac{\Z}{2}$, a contradiction. Next, we observe that $$(s_2\la ,\delta\mp\vare_3)=0\Leftrightarrow(\la ,\delta\pm\vare_2 ) =0\Leftrightarrow (s_1\la ,\delta\pm\vare_1)=0 ,~\text{ and }\langle s_i\la,(2\delta)^\vee \rangle  = \langle\la,(2\delta)^\vee \rangle,$$
 for $i=1,2.$  Therefore other cases reduce to the former case. This proves the part $(1).$

  Now, we suppose that $W_\la^1\not \cong \{e\}, D_{12}$. Since $W_\la^1$ is isomorphic to the Weyl group of a semisimple Lie algebra of smaller rank (see, e.g., \cite[Theorem 3.4]{Hu08}), it follows that  $W_\la^1$  is isomorphic to  one of the following groups
	\begin{align} \label{Eq::subgpofD12}
	&\Z_2, ~\Z_2\times \Z_2,~S_3,
	\end{align} which corresponds to Lie algebras $A_1$, $A_1\times A_1$ and $A_2$.

	 Set $\sigma: =s_2s_1$ and $\tau:=s_1$. By a direct computation, there exist $16$ subgroups of the dihedral group $W_{G_2}= \langle s_1, s_2\rangle\cong D_{12}$. The following is a list of those subgroups which are isomorphic to groups in \eqref{Eq::subgpofD12}:
 \begin{align}
 &\langle \sigma^3 \rangle,~\langle \sigma^i \tau\rangle\cong \Z_2,\\
 &\langle\sigma^3, \tau\rangle, ~\langle\sigma^3, \sigma \tau\rangle,~\langle\sigma^3, \sigma^2\tau\rangle\cong \Z_2\times \Z_2,\\
 &\langle \sigma^2, \tau\rangle,~\langle\sigma^2, \sigma\tau\rangle\cong S_3,
 \end{align}
 where  $0\leq i\leq 5$. 
 We will analyze them case by case.

 We first consider the case $W_\la^1\cong \Z_2$. Observe that $\sigma^3 = s_{\vare_1} s_{\vare_2 -\vare_3}$, which does not generate a integral Weyl group. Therefore we exclude this case. We observe that
 \begin{align} \label{Eq::Z2Gen}
 &\tau = s_1 = s_{\vare_2 -\vare_1},~\sigma \tau = s_2 = s_{\vare_1},~\sigma^2 \tau = s_{\vare_1 -\vare_3},~\sigma^3 \tau = s_{-\vare_3}, ~\sigma^4\tau = s_{\vare_2 -\vare_3},~\sigma^5\tau =s_{\vare_2}.
 \end{align} It can be checked that each element in \eqref{Eq::Z2Gen} generates an integral Weyl group of a certain weight.

 We now consider the case $W_\la^1\cong \Z_2 \times \Z_2$. Observe that \begin{align}
 &\langle\sigma^3, \tau\rangle = \langle\sigma^3\tau, \tau\rangle = \langle s_{-\vare_3}, s_{\vare_2-\vare_1}\rangle,\label{eq::38}  \\
 &\langle\sigma^3, \sigma\tau\rangle = \langle\sigma^4, \sigma\tau\rangle = \langle s_{\vare_2 -\vare_3}, s_{\vare_1}\rangle, \label{eq::39} \\
 &\langle\sigma^3, \sigma^2\tau\rangle = \langle\sigma^5\tau, \sigma^2\tau\rangle = \langle s_{\vare_2}, s_{\vare_1-\vare_3}\rangle, \label{eq::310}
  \end{align}
  Similarly, it can be checked by illustrating examples that each of groups above is an integral Weyl group of a certain weight.

  Finally, we consider $W_\la^1\cong S_3$. Suppose that $W_\la^1 =\langle \sigma^2, \tau \rangle$. We note that
  \begin{align*}
  &s_{\vare_2 -\vare_1} = \tau,~s_{\vare_1 -\vare_3} = \sigma^2 \tau,~s_{\vare_2 -\vare_3} =\tau \sigma^2 \in W_\la^1.
  \end{align*} From equations \eqref{eq::Weylgpcoroot} we find that
  \begin{align*}
  &\langle \la,\vare_1^\vee \rangle = \langle \la,(\vare_1-\vare_3)^\vee \rangle -\langle \la,(\vare_2-\vare_1)^\vee \rangle \in \Z,\\
  &\langle \la,\vare_2^\vee \rangle = \langle \la,(\vare_2-\vare_1)^\vee \rangle +\langle \la,(\vare_2-\vare_3)^\vee \rangle \in \Z,\\
  &\langle \la,(-\vare_3)^\vee \rangle = \langle \la,(\vare_1-\vare_3)^\vee \rangle +\langle \la,(\vare_2-\vare_3)^\vee \rangle \in \Z,
  \end{align*} which show that $W_\la^1 = W_{G_2}$, a contradiction.

 Now suppose that  $W_\la^1= \langle \sigma^2, \sigma\tau \rangle.$ We have $\sigma \tau = s_{\vare_1},~\sigma^3\tau = s_{-\vare_3}$, $\sigma\tau \sigma^2 = s_{\vare_2}$ and $W_\la^1 \cong S_3$.
\end{proof}

\subsubsection{}For  $H= \{e\}, ~\Z_2, ~V:=\Z_2\times \Z_2,~S_3,~W_{G_2},$ we define $\mc O_{H}$ to be the full subcategory of $\mc O$ consisting modules of weights $\la$ such that $W_\la^1\cong H.$ If $\mc O^{\mathrm{int}}$ denotes the full subcategory of modules of integral weights, then we have the following decomposition
 \begin{align}
 &\mc O = \mc O^{\mathrm{int}} \oplus \mc O_{\{e\}}\oplus \mc O_{\Z_2}\oplus \mc O_{V}\oplus \mc O_{S_3}\oplus \mc O_{W_{G_2}}.
 \end{align}

 In the remaining sections, we shall give description of blocks in $\mc O_{\{e\}}$, $\mc O_{\Z_2}$, $\mc O_{V}$, $\mc O_{S_3}$ and $\mc O_{W_{G_2}}$ and provide the complete list of tilting characters in these blocks.


\section{Typical and generic blocks} \label{Sect::4}

\subsection{Typical blocks}

We recall that a weight $\la\in \h^\ast$ called {\em strongly typical} if $\la$ is typical and $\langle \la , (2\delta)^\vee\rangle \neq 0$.  By \cite{Gor02b} a strongly typical block $\mc O_\la$ is equivalent to a block of the BGG category of the underlying Lie algebra $\mf{g}_{\bar 0}$.

Now, suppose that $\la$ is strongly typical and $\langle\la,(2\delta)^\vee\rangle\in\Z$. It follows that the character of the tilting module $\ch T_\la$ is given by the Kazhdan-Lusztig polynomial associated to the Hecke algebra of the dihedral group $D_{12}$, namely, we have
\begin{align} \label{eq::sttilting}
&\ch T_\la = \sum_{\mu \in W_\la^1 \la,~\mu\leq_{G_2}\la }\ch M_\mu,
\end{align}
where  $\leq_{G_2}$ is the Bruhat ordering on $\C\vare_1\oplus \C\vare_2$ when regarded as weight space of $G_2$.

The following theorem shows that every typical block of $G(3)$ is equivalent to a block of the underlying Lie algebra.

\begin{thm} \label{thm::tysty} Suppose that $\la\in \h^\ast$ is a typical (not necessarily  strongly typical) weight.  Then $\mc O_\la\cong \mc O_{\la -m\delta}$ for some strongly typical weight $\la -m\delta\in \h^\ast.$
\end{thm}

Before giving a proof of Theorem \ref{thm::tysty}, we shall first make a few useful observations below.

Let $(\A,\{I, \leq\})$ and $(\A',\{I, \leq\})$ be two highest weight categories (cf.~\cite{CPS88}) with simple objects $S_i, S'_i$ indexed by the same finite set $i\in I$. Then they are equivalent to some module categories of quasi-hereditary algebras. We denote by  $\Delta_i$, $T_i$ and $\Delta_i'$, $T_i'$ the standard and tilting objects in $\A$ and $\A'$ for $i\in I$, respectively. It follows from the definition and \cite[Proposition 2]{Ri91} that
\begin{align*}
&[T_i:\Delta_i] =1,~[T_i:\Delta_j] =0, \text{ unless $j\leq i$},\\
&[\Delta_i:S_i] =1,~[\Delta_i:S_j] =0, \text{ unless $j\leq i$}.
\end{align*}
The same formulas as above hold with $T_i'$, $\Delta'_i$, $S'_i$ replacing $T_i$, $\Delta_i$, and $S_i$, respectively.
With the notations above, we have the following useful lemma.
\begin{lem} \label{lem::42}
	Suppose that $\A, \A'$ admit simple-preserving dualities. Let $K: \mc \A \rightarrow \A'$ be an exact functor with an exact left adjoint $G:\A'\rightarrow \A$ satisfy the following two conditions:
\begin{enumerate}
	\item[(1)] $K$ and $G$ send the $\Delta_is$ to objects which have $\Delta$-filtrations.
	\item[(2)]    $KT_i = T_i'$, for any $i\in I$.
	\item[(3)] $[\Delta_x:S_y] = [\Delta'_x:S'_y],~[T_x:\Delta_y] = [T'_x:\Delta'_y],$ for any $x,y\in I.$
\end{enumerate}

Then $K$ and $G$ are mutually inverse equivalences of the highest weight categories $\A$ and $\A'$.
\end{lem}
\begin{proof}
   We first note that $T_j=\Delta_j =S_j$ and $T_j'=\Delta_j' =S_j'$, if  $j\in I$ is minimal in $(I,\leq)$. Therefore we have $K\Delta_j =\Delta_j'$. We proceed with induction on objects with respect to the ordering $\leq$ on $I$. Now suppose that $x\in I$ such that $K\Delta_i =\Delta_i'$, for any $i< x$.   Since $K$ is exact, we find   $K\Delta_x =\Delta'_x$ by our assumption. This shows that $K \Delta_i = \Delta_i'$, for any $i\in I$ since $I$ is a finite set.

  Again, we have $KS_j = S_j'$, if $j$ is the minimal element in $(I,\leq)$. Now suppose that $x\in I$ such that $KS_i =S_i'$ for any $i<x$. Since $K$ is exact we find that $KS_x =S_x'$ by our assumption and the fact $K\Delta_x =\Delta_x'$.

  Since $G$ is left adjoint to the  exact functor $K$, we know  that $G$ sends projective modules to projective modules. For any $i\in I$, we let $P_i$ denote the projective cover of $S_i$. By our assumption and \cite[Theorem 3.11]{CPS88} it follows that $[P_x:\Delta_y] =[P'_x:\Delta'_y]$.  For any $i,j\in I$ we have  $ \text{Hom}_{\mc O}(GP'_i, S_j) =  \text{Hom}_{\mc O}(P'_i, S_j')$ due to $KS_j=S_j'$. We may conclude that $GP_i' =P_i,$ for any $i\in I$.

  Let $j\in I$ be a maximal element.  Then $P_j = \Delta_j$, $P'_j =\Delta'_j$ and so  $G\Delta'_j =\Delta_j$.
  Suppose that $x\in I$ such that $G\Delta'_i =\Delta_i$ for any $i\in I$ with $x<i$.  The exactness of $G$, the fact that $[P_x:\Delta_i] = [P_x':\Delta_i']$ and $[P_x:\Delta_x] =1$ imply that $G\Delta'_x =\Delta_x$. Consequently,  $G\Delta'_i =\Delta_i$, for any $i\in I$ since $I$ is a finite set.  In particular, $GS'_j =S_j$ if $j$ is minimal in $I$. Using a similar argument as above, we may conclude that $GS'_i =S_i$, for any $i\in I.$

  Finally, let $\eta: 1_{\A'}\rightarrow KG$ and $\vare:GK \rightarrow 1_{\A}$ be respectively the unit and counit of the adjoint pair $(G,K)$. We claim that $\eta$ and $\vare$ are isomorphism of functors. By the  definition of unit and counit, $KS_i\neq 0$ and $GS_i'\neq 0$, which imply that $\eta_{S'_i}: S'_i \rightarrow KGS'_i$ and $\vare_{S_i}: GKS_i\rightarrow S_i$ are non-zero and so they are isomorphisms, for any $i\in I$. Now a standard argument using the Short Five Lemma completes the proof.
\end{proof}

\begin{proof}[Proof of Theorem \ref{thm::tysty}]
By \cite{Gor02b} it suffices to consider the case when $\la$ is typical, but not strongly typical, which means that $\la$ is of the form
	\begin{align}\label{typ:wt:1}
	&\la = a\omega_1 +b\omega_2,
	\end{align} for some $a,b\in \C.$	
	We first claim that \eqref{eq::sttilting} still holds in this case.
For any $m>0$ we set $L(m\delta):=L_{m\delta+\rho}$ to be the irreducible modules of highest weight $m\delta$.  We now choose $m\ge 2$ so that $L(m\delta)$ is a finite-dimensional module. In addition we choose $m$ in such a way that the following condition also holds: $|\langle \la ,\alpha^\vee\rangle|<2m$, for any short root $\alpha \in \Phi^+_{\ov 0}$ (i.e., $|b|, |3a+b|,|3a+2b|<2m$).
Then the weight $\la -m\delta$ is strongly typical by \eqref{eq::wtid}.

Denote by $F(-):\mc O_{\la -m\delta} \rightarrow \mc O_\la$ the usual translation functor. Since $\ch L(m\delta)$ is invariant under $W$, we find that if $\nu$ is a weight of $L(m\delta)$ with $\langle \nu, (2\delta)^\vee \rangle =m$, then $\nu =2m\delta$.  Since $\la$ is typical we have \[\ch FT_{\la -m\delta} = \sum_{\eta \in W_\la^1 \la-m\delta,~\eta\leq_{G_2}\la-m\delta }\ch F M_{\eta} =  \sum_{\mu \in W_\la^1 \la,~\mu\leq_{G_2}\la }\ch M_{\mu}. \]
Finally, we note $[T_\la: M_{\mu}]>0$, for $\mu \in W_\la^1 \la,~\mu\leq_{G_2}\la  $, by Proposition \ref{prop:flags}, and consequently,  $FT_{\la-m\delta} = T_\la$. This proves \eqref{eq::sttilting} for $\la$ of the form \eqref{typ:wt:1}.

Now we apply Lemma \ref{lem::42} to complete the proof.
\end{proof}

  \subsection{Generic Blocks}
   We shall show in this section that $\mc O_{\{e\}}$ is equivalent to a direct sum of the principal blocks of $\gl(1|1)$.

 A weight $\la \in \h^\ast$ is said to be {\em generic} if $W_\la^1$ is trivial, namely, $\langle \la,\gamma^\vee\rangle\notin \Z$ for any  $\gamma \in \Phi^+_{\oa, \circ}$.  It follows from Theorem \ref{thm::desnonblocks} (i) that if $\la$ is atypical and generic, then $W_\la$ is trivial. 

 \begin{prop}
 	Suppose that $\la \in \h^\ast$ is an atypical and  generic weight with  $\alpha\in A(\mu)$.
 	Then $\emph{Irr} \mc O_\la =\{{\la +k\alpha}|~k\in \Z\}$. Furthermore, we have $\ch M_{\la+k\alpha} = \ch L_{\la +k\alpha} + \ch L_{\la +(k-1)\alpha},$ for any $k \in \Z$.
 \end{prop}
 \begin{proof}	Let $k\in \Z$.
  By      Lemma \ref{lem::11} we have     $A(\la +k\alpha) =\{\alpha\}.$ Therefore it follows from  Proposition~\ref{prop::JSF} that
\begin{align}\label{eq:JSF}
 	\sum_{i\ge 1} \ch M^i_{\la+k\alpha} =\frac{\ch M_{\la  +(k-1)\alpha}}{1+e^{-\alpha}}.
\end{align}
Since $\frac{\ch M_{\la +k\alpha }}{1+e^{-\alpha}}$ is the character of a $\mf g$-module by \cite[Theorem 1.10]{Mu17} for every $k$, it follows that $\ch M_{\la+k\alpha} =\frac{\ch M_{\la +k\alpha }}{1+e^{-\alpha}}+\frac{\ch M_{\la +(k-1)\alpha }}{1+e^{-\alpha}}$ is the sum of the characters of two $\mf g$-modules. Since $M_{\la+k\alpha}^1$ is the radical of $M_{\la+k\alpha}$, it follows that $\frac{\ch M_{\la +(k-1)\alpha }}{1+e^{-\alpha}}\le \ch M_{\la+k\alpha}^1$. (Here  $\le$ is understood in an obvious way.) However, \eqref{eq:JSF} implies that $\ch M_{\la+k\alpha}^1\le\sum_{i\ge 1} \ch M^i_{\la+k\alpha}=\frac{\ch M_{\la +(k-1)\alpha }}{1+e^{-\alpha}}$. Thus, $\ch M_{\la+k\alpha}^1=\frac{\ch M_{\la +(k-1)\alpha }}{1+e^{-\alpha}}$ and so $\ch L_{\la+k\alpha}=\frac{\ch M_{\la +k\alpha }}{1+e^{-\alpha}}$.
 \end{proof}


\section{Reduction methods and characters in  $\mc O_{\Z_2}$} \label{Sect::1intblocks}

\subsection{Some Levi subalgebras} \label{sect::51} We recall the simple systems $\Pi^i$, their corresponding Borel subalgebras $\mf b^i$ and BGG cateogries $\mc O^{(i)}$ from Sections \ref{subsect::simplesys} and \ref{subsect::24}. In this section, we will  investigate some Levi subalgebras of $\mf g$.

For a given subset $S\subseteq \Pi^i$, recall that $\mf l_S$ denotes the Levi subalgebra given by $S$ as defined in Section \ref{subsect::24}. That is, $\mf l_S$ is generated by $\mf h$ and root vectors of roots $\pm \alpha$ for all $\alpha \in S$.

The following lemma follows from Cartan matrices of contragredient Lie superalgebras, see, e.g., \cite[Section 2.5]{Kac77} and \cite[Appendix B]{Mu17}.
\begin{lem} \label{lem::52} With the notations as above, we have the following isomorphisms:
	\begin{enumerate}
		\item[(1)] Let $S:=\{\vare_2 -\vare_1,~\delta +\vare_3\}\subset \Pi$. Then $\mf l_S\cong \mf{sl}(2)\oplus \mf{gl}(1|1)$.
		\item[(2)] Let  $S$ be either $\{\vare_2 -\vare_1,~\delta -\vare_2 \} \subset \Pi^2$, or $\{\vare_1,~ \delta +\vare_3\} \subset \Pi$. Then  $\mf l_S\cong \gl(2|1)$.
		\item[(3)] Let $S:=\{\delta -\vare_1,~\vare_1\}\subset \Pi^2$. Then $\mf l_S \cong \mf{osp}(3|2)$.
	\end{enumerate}
\end{lem}

We refer the reader to \cite{CW08} and, respectively, Appendix \ref{sect::app} for the irreducible characters of $\gl(2|1)$ and $\mf{osp}(3|2)$. In this section, we denote by $\mc O^{\mf l}_{\pri}$ the principal block of $\mc O^{\mf l}$ in the case when $\mf l=\mf l_S$ as above.

 \subsection{Blocks of $\mc O_{\Z_2}$}
In this section we classify blocks of $\mc O_{\Z_2}$.

 \begin{thm} \label{Thm::1-intblocks}
 	Let $\la\in \h^\ast$ such that $W_\la^1 \cong \Z_2$. Suppose that $\alpha \in A(\la)$ and $\gamma \in Z(\la)$. Then  $W_\la \cong \Z_2$ or $W_\la=\Z_2\times \Z_2$. Furthermore, we have the following classification:
 	
 	\begin{enumerate}
 		\item[(1)] Suppose that $\gamma$ is a long root.
 		\begin{enumerate}
 			\item[(a)]If $(\gamma, \alpha) \neq 0$, then $W_\la \cong\Z_2$ and $\mc O_\la \cong \mc O_\pri^{\gl(2|1)}$.
 			\item[(b)]  If $(\gamma, \alpha) =0$, then $W_\la \cong\Z_2$ and $\mc O_\la \cong \mc O_\pri^{\mf{sl}(2) \oplus \gl(1|1)}$.
 		\end{enumerate}
 		
 		\item[(2)]   Suppose that $\gamma$ is a short root.
 		\begin{enumerate}
 			\item[(c)]If $(\gamma, \alpha) = \pm 2 $ with $\langle \la , (2\delta)^\vee \rangle \in \frac{1}{2} +\Z$, then $W_\la   \cong \Z_2 \times \Z_2$ and $\mc O_\la \cong \mc O_\mu^{\mf{osp}(3|2)}$, for some integral weight $\mu$ of $\mf{osp}(3|2)$.
 			\item[(d)] If $(\gamma, \alpha) = \pm 2 $ with $\langle \la , (2\delta)^\vee \rangle \notin \frac{1}{2} +\Z$, then $W_\la  \cong \Z_2$ and $\mc O_\la \cong \mc O_\mu^{\mf{osp}(3|2)}$, for some non-integral weight $\mu$ of $\mf{osp}(3|2)$.
 			\item[(e)]  If $(\gamma, \alpha) =\pm 1$. Then $W_\la  \cong \Z_2$ and $\mc O_\la \cong \mc O_\pri^{ \gl(2|1)}.$
 		\end{enumerate}
 	\end{enumerate}
 \end{thm}

 We will  in the remaining sections provide a proof of Theorem \ref{Thm::1-intblocks}.

 \subsection{Equivalence of twisting functors}
 For a given atypical weight $\la\in \h^\ast$ and $i=1,2$, the following diagram
 \begin{align}
 &[A(\la)|Z(\la)]  \xrightarrow{s_i} [s_iA(\la)| Z(s_i\la)], \label{gooddiagram}
 \end{align} is called \emph{good} if $\alpha_i\notin Z(\la).$ We note   the following  useful facts \begin{align}
 &A(s_i \la) =s_i A(\la), \label{eq::of1}\\
 &Z(s_i\la) =\{\gamma \in \Phi_{\oa, o}^+| ~\pm\gamma \in s_i Z(\la) \},\label{eq::of2}\\
 &(\gamma,\alpha) = (s_i\gamma,s_i\alpha), ~\text{for any }\alpha, \gamma.\label{eq::of3}
 \end{align}
 The definition of good diagram is motivated by the following lemma.

 \begin{lem} \label{lem::tequiv} If \begin{align*}
 	&[A(\la)|Z(\la)]  \xrightarrow{s_i} [s_iA(\la)| Z(s_i\la)],
 	\end{align*} is a good diagram, then $\mc O_\la$ and $\mc O_{s_i\la}$ are equivalent as highest weight categories.
 \end{lem}
 \begin{proof}
 	Let $T_{s_i}$ be the twisting functor corresponding to $s_i$. We note that $\alpha_1$, $\alpha_2$ are also simple in the root system of $G_2$, therefore  the star action coincides with the dot action for (non-shiffted) weights (see, e.g., \cite[Section 8]{CM16}). By \cite[Proposition 8.6]{CM16}, $T_{s_i}: \mc O_\la \rightarrow \mc O_{s_i \la}$ is an equivalence of categories sending a Verma module of highest weight $\mu$ to a Verma module of highest weight $s_i \mu$. This completes the proof.
 \end{proof}

 \begin{lem} \label{lem::1} Let $\la \in   \h^\ast$ such that $W_\la^1\cong \Z_2$ with  $\gamma\in Z(\la)$ a long root. Then there is $w\in W$ such that $\mc O_\la \cong \mc O_{w\la}$ with $Z(w\la) =\{\vare_2-\vare_1\}$.
 \end{lem}
 \begin{proof}
 	Using Lemma \ref{lem::tequiv}, the proof follows from  the following two possibilities of good diagrams
 	\begin{align}
 	&[A(\la)|\vare_1 -\vare_3]  \xrightarrow{s_2} [s_2A(\la)| \vare_2 - \vare_1],\\
 	&[A(\la)|\vare_2 -\vare_3]  \xrightarrow{s_1} [s_1A(\la)| \vare_1 - \vare_3].
 	\end{align}
 \end{proof}

 \begin{lem} \label{lem::2} Let $\la \in  \h^\ast$ and suppose that $\gamma\in Z(\la)$ a short root. Then there is $w\in W$ such that $\mc O_\la \cong \mc O_{w\la}$ with $Z(w\la) =\{-\vare_3\}$.
 \end{lem}
 \begin{proof}  We have the following good diagrams
 	\begin{align}
 	&[A(\la)|\vare_2]  \xrightarrow{s_2} [s_2A(\la)| -\vare_3], \\
 	&[A(\la)|\vare_1]  \xrightarrow{s_1} [s_1A(\la)| \vare_2].
 	\end{align}  The proof follows from Lemma \ref{lem::tequiv}.
 \end{proof}

\subsection{Proof of Theorem \ref{Thm::1-intblocks}}  By Lemmas \ref{lem::1}, \ref{lem::2} and \eqref{eq::of1}--\eqref{eq::of3}, we have reduced the proof of Theorem \ref{Thm::1-intblocks} to two cases:  $Z(\la) =\{\vare_2-\vare_1\}$ and $Z(\la) =\{-\vare_3\}$. Also, we recall that  the simple systems $\Pi^i$, Borel subalgebras $\mf b^i$ and the corresponding BGG categories $\mc O^{(i)}$ from Sections \ref{subsect::simplesys} and \ref{subsect::24},  for $0\leq i\leq 3$.

 \subsubsection{{\bf Case 1}: $Z(\la) =\{\vare_2-\vare_1\}$}\label{subsubsect:aux1}

 We assume in this Section \ref{subsubsect:aux1} that $\la\in \h^\ast$ such that $W_\la^1\cong \Z_2$   with $$Z(\la) =\{\gamma:=\vare_2-\vare_1\}.$$
 Before describing the blocks of this type, we need the following.

\begin{itemize}\label{item::441}
	\item[(a)] Assume that there is  $ \alpha := \delta +c\vare_3 \in A(\la)$, for   $c\in \{\pm 1\}$.   By Theorem \ref{thm::blocks} we have  $A(\la) = \{\alpha\}$ and $$\text{Irr}\mc O_\la =\{\la+ k\alpha, s_1\la +k\alpha|~k\in \Z\}.$$  Also, in this case we have $(\gamma,\alpha)=0.$
		\item[(b)]
		Assume that there is an $i\in \{1,2\}$ and $ c=\{\pm 1\}$ such that $ \delta +c\vare_i \in A(\la)$. Using  Proposition \ref{prop::JSF} to replace $\la$ by $\la +(-1)^{i+1}c\langle \la ,\gamma^\vee \rangle (\delta +c\vare_i)$, if necessarily, we can assume that $\langle \la, \gamma^\vee \rangle =0$. Let $\alpha =\delta +c\vare_1$ and $\beta =\delta +c\vare_2$, then $A(\la) = \{\alpha, \beta\}$ and $$\text{Irr}\mc O_\la =\{\la+ k\alpha, \la +k\beta|~k\in \Z\}$$ by Theorem \ref{thm::blocks}.  Also, in this case we have $(\gamma,\alpha)\neq 0, (\gamma,\beta)\neq 0.$
\end{itemize}

 \begin{lem}
 	Assume that $(\gamma,\alpha)=0$. Then  $W_\la =\langle s_1\rangle$ and $\mc O_\la \cong \mc O_\pri^{\mf{sl}(2) \oplus \gl(1|1)}$.
 \end{lem}
 \begin{proof}
 	In this case, we have 	$A(\la) = \{\alpha:= \delta+c \vare_3\}$, for   $c=\pm 1.$
 	
 	We first consider $A(\la) = \{\alpha:= \delta+ \vare_3\}$. Then $\text{Irr}\mc O_\la =\{\la +k\alpha, s_1\la +k\alpha|~k\in \Z\}$.  We take the Levi subalgebra $\mf l:=\mf l_S\cong \mf{sl}(2)\oplus \mf{gl}(1|1)$ with $S:=\{\vare_2 -\vare_1, \delta +\vare_3\} \subset \Pi$ from Lemma \ref{lem::52}. Note that $s_1 \la +k\alpha \in \la +\Z (\vare_2 -\vare_1) +\Z \alpha$. Then the parabolic induction functor $\text{Ind}_{\mf l+\mf b}^{\mf g}: \mc O^{\mf l}_{\la} \rightarrow \mc O_\la$ sends simple objects to simple objects. By standard arguments using adjunction,
 	induction on the length of a module, and the Short Five Lemma, see, e.g., the proof of \cite[Proposition 3.6]{CMW13}, it follows that $\text{Ind}_{\mf l+\mf b}^{\mf g}: \mc O^{\mf l}_\la\rightarrow \mc O_\la$ is an equivalence of categories. We may note that $\mc O_\la^{\mf l}\cong \mc O_\pri^\mf{l}$.  	

 	Next, we consider $A(\la) = \{  \delta- \vare_3\}$. In this case, we have the following sequence of good diagrams
 	\begin{align*}
 	&[\delta- \vare_3|\vare_2 -\vare_1]  \xrightarrow{s_2} [\delta+\vare_2| \vare_1 - \vare_3]\xrightarrow{s_1} [\delta+\vare_1| \vare_2 - \vare_3]\\
 	&\xrightarrow{s_2} [\delta-\vare_1| \vare_2 - \vare_3]\xrightarrow{s_1} [\delta-\vare_2| \vare_1 - \vare_3]\xrightarrow{s_2} [\delta+\vare_3| \vare_2 - \vare_1].
 	\end{align*}  By Lemma \ref{lem::tequiv}, the proof of this case now reduces to the former case. The conclusion follows.
 \end{proof}

 \begin{lem}
 	Assume that $(\gamma,\alpha)\neq 0$. Then $W_\la =\langle s_1 \rangle$ and $\mc O_\la \cong \mc O_\pri^{\gl(2|1)}$.
 \end{lem}
 \begin{proof}
 	Observe that we have 	$A(\la) = \{ \delta+ c\vare_1,  \delta +c\vare_2\}$, for some $c=\pm 1$. In this case, we may assume that $\langle \la, (\vare_2 -\vare_1)^\vee\rangle =0$ by the analysis in \eqref{item::441}(b).
 	
 	We first consider the case $A(\la) = \{\alpha:= \delta- \vare_1,~\beta:= \delta -\vare_2\}$. Then $\text{Irr}\mc O_\la =\{\la +k\alpha, \la +k\beta|~k\in \Z\}$.  
 	We note that $\mf b^1$ is obtained from $\mf b$ by applying odd reflection with respect to $\delta+\vare_3$.  By the rule of changing highest weights under odd reflection we find that $\text{Irr}\mc O^{(1)}_\la =\text{Irr}\mc O_\la$.   We take the Levi subalgebra $\mf l :=\mf l_S\cong \gl(2|1)$ with  $S:=\{\vare_2 -\vare_1,~\delta -\vare_2\} \subset \Pi^1$ from Lemma \ref{lem::52}. Then parabolic induction functor $\text{Ind}_{\mf l+\mf b^1}^{\mf g}: (\mc O^{(1)})^{\mf l}_\la \rightarrow \mc O^{(1)}_\la$ sends simple objects to simple objects, and is thus an equivalence, as explained above.

 	Next, we  consider the case $A(\la) = \{\alpha:= \delta+ \vare_1,~\beta:= \delta +\vare_2\}$. Then $\text{Irr}\mc O_\la =\{\la +k\alpha, \la +k\beta|~k\in \Z\}$. In this case, we  apply the following good sequence of diagrams
 	\begin{align*}
 	&[\delta+\vare_1, ~\delta+\vare_2|\vare_2 -\vare_1]  \xrightarrow{s_2} [\delta-\vare_1, ~\delta-\vare_3|\vare_1 -\vare_3]  \xrightarrow{s_1}[\delta-\vare_2, ~\delta-\vare_3|\vare_2 -\vare_3] \\
 	&\xrightarrow{s_2} [\delta+\vare_3, \delta+\vare_2| \vare_2 - \vare_3]\xrightarrow{s_1} [\delta+\vare_3, \delta+\vare_1| \vare_1 - \vare_3]\xrightarrow{s_2} [\delta-\vare_2, \delta-\vare_1| \vare_2 - \vare_1].
 	\end{align*}  By Lemma \ref{lem::tequiv}, the proof of this case now reduces to the former case.
 \end{proof}

 \subsubsection{{\bf Case 2}: $Z(\la) =\{-\vare_3\}$}\label{subsubsect:aux2}

 We assume in this Section \ref{subsubsect:aux2} that $\la\in \h^\ast$ with $$Z(\la) =\{\gamma:=-\vare_3\}.$$ Before describing the blocks of this type, we need the following.

\begin{itemize} \label{item::442}
	\item[(a)] Assume that    $ \alpha := \delta +c\vare_3 \in A(\la)$ with $c = \pm 1$ and suppose that $\langle \la, (2\delta)^\vee\rangle \in \frac{1}{2} +\Z$.  Let  $\beta :=\delta -c\vare_3$.   By  Theorem \ref{thm::blocks} we have  $$\text{Irr}\mc O_\la =\{\la+ k\alpha, s_{2\delta}\la +k\beta, s_{\gamma}\la+k\beta, s_{2\delta}s_{\gamma}\la +k\alpha|~k\in \Z\}.$$ Also, in this case we have $(\gamma,\alpha), (\gamma,\beta)=\pm 2.$
	
	\item[(b)]Assume that there is  $\alpha := \delta +c\vare_3 \in A(\la)$ with $c =\pm 1$ and suppose that $\langle \la, (2\delta)^\vee\rangle \notin \frac{1}{2} +\Z$.  Let  $\beta :=\delta -c\vare_3$. By Proposition \ref{prop::JSF}, replacing $\la$ by $\la -\langle\la , (2\delta)^\vee \rangle \alpha$ if necessarily, we may assume that $\langle \la,\vare_3^\vee \rangle =0$. By  Theorem \ref{thm::blocks} we have  $$\text{Irr}\mc O_\la =\{\la+ k\alpha, \la +k\beta|~k\in \Z\}.$$  Also, in this case we have $(\gamma,\alpha), (\gamma,\beta)=\pm 2.$
	
	\item[(c)] Assume that there is $i\in \{1,2\}$ and $ c=\pm 1$ such that $ \delta +c\vare_i \in A(\la)$. By Proposition \ref{prop::JSF}, replacing $\la$ by $\la -c\langle \la ,\gamma^\vee \rangle (\delta +c\vare_i)$ if necessarily, we may assume that $\langle \la, \vare_3^\vee \rangle =0$. Let $\alpha =\delta +c\vare_i$ and $\beta =\delta -c\vare_j$  with $\{i, j\} =\{1,2\}$.  By Theorem \ref{thm::blocks} we have  $$\text{Irr}\mc O_\la =\{\la+ k\alpha, \la +k\beta|~k\in \Z\}.$$  Also, in this case we have $(\gamma,\alpha), (\gamma,\beta)=\pm 1.$
\end{itemize}

 \begin{lem} \label{lem::1_1} Suppose that $(\gamma, \alpha) = \pm 2 $ with $\langle \la , (2\delta)^\vee \rangle \in \frac{1}{2} +\Z$. Then $W_\la  = \langle s_{0}, s_{\gamma}\rangle$ and $\mc O_\la \cong \mc O_\mu^{\mf{osp}(3|2)}$, for some integral weight $\mu$ of $\mf{osp}(3|2)$.
 \end{lem}
 \begin{proof} Under the hypothesis of the lemma we are in the case of \ref{item::442}(a). We have $ \alpha := \delta +c\vare_3 \in A(\la)$,  $c = \pm 1$,  $\beta :=\delta -c\vare_3$, and  $$\text{Irr}\mc O_\la =\{\la+ k\alpha, s_{0}\la +k\beta, s_{\gamma}\la+k\beta, s_{0}s_{\gamma}\la +k\alpha|~k\in \Z\}.$$
 	
 	We apply the following good diagrams to obtain an equivalence $\mc O_\la \cong \mc O_{s_1s_2\la}:$
 	\begin{align*}
 	&[\delta +c\vare_3|-\vare_3]\xrightarrow{s_2} [\delta - c\vare_2|\vare_2] \xrightarrow{s_1} [\delta -c\vare_1|\vare_1],
 	\end{align*} and the corresponding equivalence give rises to $\text{Irr}\mc O_{s_1s_2\la} = s_1s_2\text{Irr} \mc O_\la = W_{s_1s_2\la}(s_1s_2\la+k(\delta-c\vare_1))$ by Theorem \ref{thm::blocks}. 
 	 We now set $\mu :=s_1s_2 \la$.  
 	 	By the rule of changing highest weights under odd reflection we find that $\text{Irr}\mc O^{(2)}_{\mu} =\text{Irr}\mc O_{\mu}$. Therefore  $\mc O_{\mu} \cong \mc O^{(2)}_{\mu}$ via the identity functor. We now take the Levi subalgebra $\mf l :=\mf l_S  \cong \mf{osp}(3|2)$ with $S:=\{\delta -\vare_1,\vare_1\}\subset \Pi^{(2)}$. Now, the parabolic induction functor $\text{Ind}_{\mf l+\mf b^2}^{\mf g}: (\mc O^{(2)})^{\mf l+\mf b^2}_{\mu} \rightarrow \mc O^{(2)}_{\mu}$, sends simple objects to simple objects, is thus an equivalence.
 \end{proof}

 \begin{lem}  \label{lem::1_2}
 	Suppose that $(\gamma, \alpha) = \pm 2 $ with $\langle \la , (2\delta)^\vee \rangle \notin \frac{1}{2} +\Z$. Then $W_\la  \cong \Z_2$ and $\mc O_\la \cong \mc O_\mu^{\mf{osp}(3|2)}$, for some non-integral weight $\mu$ of $\mf{osp}(3|2)$.
 \end{lem}
 \begin{proof} Under the hypothesis of the lemma we are in the case \ref{item::442}(b) above and we have
 	$A(\la) = \{\alpha := \delta +\vare_3, \beta := \delta -\vare_3\}$ and $$\text{Irr}\mc O_\la =\{\la+ k\alpha, \la +k\beta|~k\in \Z\}.$$
 	
 	The proof can be completed by using similar argument as in the proof of Lemma \ref{lem::1_1}.
 \end{proof}

\begin{lem}
Suppose that $(\gamma, \alpha) =\pm 1$. Then $W_\la  \cong \Z_2$ and $\mc O_\la \cong \mc O_\pri^{ \gl(2|1)}.$
\end{lem}
\begin{proof}
Under the hypothesis of the lemma we are in the case of \ref{item::442}(c), and so explained there we can assume that $\langle \la, \gamma^\vee\rangle =0$ and that 	$A(\la) = \{  \delta+ c\vare_1,   \delta -c\vare_2\}$, for $c=\pm 1$.
 	
We first consider the case $A(\la) = \{\alpha:= \delta+\vare_1,~\beta:= \delta -\vare_2\}$.
In this case, we have the following composition of good diagrams:
\begin{align*}
&[\delta+\vare_1, ~\delta-\vare_2| -\vare_3]  \xrightarrow{s_2} [\delta-\vare_1, ~\delta+\vare_3|\vare_2]  \xrightarrow{s_1}[\delta-\vare_2, ~\delta+\vare_3|\vare_1].
\end{align*}
By Lemma \ref{lem::tequiv}, we have $\mc O_\la \cong \mc O_{s_1s_2\la}$ and $A(s_1s_2\la) = \{\delta - \vare_2, \delta +\vare_3\}$ with $\text{Irr}\mc O_{s_1s_2\la} =\{s_1s_2\la +k(\delta- \vare_2), s_1s_2\la +k(\delta +\vare_3)|~k\in \Z\}$.  Now,  we take the Levi subalgebra $\mf l:=\mf l_S   \cong \gl(2|1)$ with  $S:=\{\vare_1,\delta +\vare_3\}\subset \Pi$. Then the parabolic induction functor $\text{Ind}_{\mf l+\mf b}^{\mf g}: \mc O^{\mf l}_\la \rightarrow \mc O_\la$, sends simple objects to simple objects, is thus an equivalence. We note that $\mc O^{\mf l}_\la\cong \mc O^{\mf l}_\pri$.

 	Next, we  consider the case $A(\la) = \{\alpha:= \delta-\vare_1,~\beta:= \delta +\vare_2\}$. Then $\text{Irr}\mc O_\la =\{\la +k\alpha, \la +k\beta|~k\in \Z\}$. In this case, we  apply the following  sequence of good diagrams
 	\begin{align*}
 	&[\delta-\vare_1, ~\delta+\vare_2|-\vare_3]  \xrightarrow{s_1} [\delta-\vare_2, ~\delta+\vare_1|-\vare_3]  \xrightarrow{s_2}[\delta+\vare_3, ~\delta-\vare_1|\vare_2]
 	\xrightarrow{s_1}[\delta+\vare_3, ~\delta-\vare_2|\vare_1].
 	\end{align*}  We use Lemma \ref{lem::tequiv} to reduce  the proof of this case to the previous case.
 \end{proof}

 \begin{proof}[Proof of Theorem \ref{Thm::1-intblocks}]
 	The proof of Theorem \ref{Thm::1-intblocks} now follows from \eqref{eq::of1}--\eqref{eq::of3} and the lemmas in Section \ref{Sect::1intblocks}.
 \end{proof}


\section{Character formulas in $\mc O_V$} \label{sect::OVcase}

\subsection{Blocks in $\mc O_V$  cases}

Recall that $V=\Z_2\times \Z_2$. In this section, we give a description of blocks $\mc O_\la$ for atypical non-integral weights $\la\in \h^\ast$ satisfying $W_\la^1 \cong V$.

In   Section \ref{subsect::Vblocksspecial} we  describe $\mc O_\la$ in the case when $Z(\la)=\{-\vare_3,~\vare_2-\vare_1\}$ then   show in the Section \ref{subsect::Vblocksreduction} that all blocks are equivalent to such blocks.

\subsubsection{The case $Z(\la) =\{-\vare_3,~ \vare_2-\vare_1\}$} \label{subsect::Vblocksspecial}
We suppose in this section that $\la=d\delta+a\omega_1+b\omega_2 \in \h^\ast$ with  $Z(\la) =\{-\vare_3,~\vare_2-\vare_1\}$. By   \eqref{eq::Weylgpcoroot} we have $a,~3a+2b\in \Z$ and $b,~a+b,~2a+b,~3a+b \not \in \Z$. This implies that $b={k}/{2}$, where $k$ is odd. In particular, we can rewrite $$\la = [d~|~{k}/{4},~ {3a}/{2}+{k}/{4},~-{3a}/{2}-{k}/{2}].$$

If $\la$ is typical, then $\mc O_\la \cong \mc O_\la^{\mf{sl}(2)\oplus G_2}$ and hence we have $\mc O_\la\cong  \mc O^{\mf{sl}(2)\oplus\mf{sl}(2)}_0$ or {{$\mc O_\la\cong\mc O^{\mf{sl}(2)\oplus\mf{sl}(2)\oplus\mf{sl}(2)}_0$.}}

We suppose now that $\la$  is atypical. By our assumption  we have $W_\la \cong V$ or $\Z_2 \times V$.   It follows from \cite[Proposition 3.2]{CW18} that
\begin{align} \label{eq::eqV}
&d= \pm{k}/{4},~\pm\left({3a}/{2}+{k}/{4}\right), \text{ or } \pm\left({3a}/{2}+{k}/{2}\right).
\end{align}

We define the following weights:
\begin{itemize}
	\item[(a)] For  $\ell \in 3\Z$, we define
	\begin{align}
	&\la_{[\ell]}:= [-\ell - {1}/{2}~|~{1}/{4},~\ell+{1}/{4},~ -\ell -{1}/{2}].
	\end{align}
	\item[(b)]
	For  $\ell \in {1}/{4}+{\Z}/{2}$, we define
	\begin{align}
	&\mu_{[\ell]}:= [\ell~|~\ell,~\ell,~ -2\ell].
	\end{align}
		\item[(c)]
	For  $\ell \in 3\Z+1$, we define
	\begin{align}
	&\nu_{[\ell]}:= [\ell~|~{1}/{4},~-\ell-{1}/{4},~ \ell].
	\end{align}
\end{itemize}

\vskip 0.5cm

We now give the classification of indecomposable  summand of $\mc O_{V}$ as follows.
\begin{thm} \label{thm::51}Let $\la$ be  atypical with $Z(\la)=\{-\epsilon_3,\epsilon_2-\epsilon_1\}$.	We note that $s_{-\vare_3} =s_2s_1s_2s_1s_2$.

	 \vskip 0.2cm
	 \begin{enumerate}
	 	\item[(I)]	The following are equivalent.
	 	\begin{itemize}
	 		\item[(1)] $W_\la = \langle s_0, s_{1}, s_{-\vare_3}\rangle \cong \Z_2\times \Z_2\times \Z_2$.
	 		\item[(2)] $\langle\la,(2\delta)^\vee\rangle \in \frac{1}{2}+\Z$ and $\delta+c\vare_3 \in A(\la)$,   $c= \pm 1.$
	 		\item[(3)] $\la\in \emph{Irr}\mc O_{\la_{[\ell]}},$ for some $\ell \in 3\Z$.
	 	\end{itemize}
 	\item[(II)]	The following are equivalent.
 	\begin{itemize}
 		\item[(4)] $W_\la=\langle  s_{1}, s_{-\vare_3}\rangle \cong \Z_2\times \Z_2$ with $(\alpha,\gamma)$  odd, for any $\alpha \in A(\la)$ and $\gamma\in Z(\la)$.
 		\item[(5)]  $\delta+c\vare_i \in A(\la)$,  $c= \pm 1$, $i=1,2$.
 		\item[(6)] $\la\in \emph{Irr}\mc O_{\mu_{[\ell]}},$ for some $\ell \in  \frac{1}{4}+\frac{\Z}{2}$.
 	\end{itemize}
 \item [(III)]	The following are equivalent.
 \begin{itemize}
 	\item[(7)] $W_\la =\langle  s_{1}, s_{-\vare_3}\rangle \cong \Z_2\times \Z_2$ with $(\alpha,\gamma)$  even, for any $\alpha \in A(\la)$ and $\gamma\in Z(\la)$.
 	\item[(8)] $\langle\la,(2\delta)^\vee\rangle \notin \frac{1}{2}+\Z$ and $\delta+c\vare_3 \in A(\la)$,  $c= \pm 1.$
 	\item[(9)] $\la\in \emph{Irr}\mc O_{\nu_{[\ell]}},$  $\ell \in \Z$.
 \end{itemize}
	 \end{enumerate}
\end{thm}
\begin{proof}
	We first consider the case that $W_\la \cong \Z_2\times \Z_2\times \Z_2$. We claim that $(1)$ and $(3)$ are equivalent. Note that we have $(3) \Rightarrow (1)$. It remains to show  $(1) \Rightarrow (3)$. To see this, we note that  $\la = [d|~\frac{k}{4},~\frac{3a}{2}+\frac{k}{4},~-\frac{3a}{2}-\frac{k}{2}]$ with $d\in \frac{1}{2}+\Z$. Since $k$ is odd, we have $d\neq \frac{k}{4}$. Since $3a+b\not\in\Z$, $\frac{3a}{2}+\frac{k}{4}\not \in \frac{1}{2}+\Z.$ Therefore, $-\frac{3a}{2}-\frac{k}{2} \in \frac{1}{2}+\Z$. This means that $a$ is an even integer. We set $\ell : = \frac{3a}{2},$ which is an integer and lies in $3\Z$. Therefore, applying $s_0$ if necessarily, we may assume that
	\begin{align}
	&\la = [-\ell -{k}/{2}~|~{k}/{4},~\ell+{k}/{4},~-\ell -{k}/{2}].
	\end{align}
	We note that $\la =\la_{[\ell]} +({1-k}/{2})[1|~ -{1}/{2},~-{1}/{2},~1]= \la_{[\ell]} +({1-k}/{2})(\delta+\epsilon_3).$ By Theorem \ref{thm::blocks} we have thus proved $(1)\Leftrightarrow (3)$ and $(1)\Rightarrow (2)$ as well. Also, it is clear  that $(2)\Rightarrow (1)$ since we have assumed that $W_\la^1\cong V$.
	
	We next consider  the case that $W_\la \cong  V$. Choose arbitrary $\alpha\in A(\la)$, $\gamma\in Z(\la)$ and note that
	$$(\gamma,\alpha)= \begin{cases}
	\text{ even},& \text { if $\alpha = \delta\pm \vare_3$;}\\
		\text{ odd},& \text { if $\alpha = \delta\pm \vare_i$ with $i=1,2.$}
	\end{cases}$$

	We first suppose that $(\gamma, \alpha)$ is odd. Let us recall $d,a,b$ from \eqref{eq::wtid}. By \cite[Propposition 3.2]{CW18}, we have either $d=\pm \frac{b}{2}$, or $d=\pm \frac{3a+b}{2}$. Since $b\notin \Z$ and $3a+b\notin \Z$, we have $d\notin \frac{\Z}{2}$. This proves $(5)\Leftrightarrow (4).$
	
We now prove $(4)\Leftrightarrow (6)$. Observe that we have $(6)\Rightarrow (4)$. It remains to prove $(5)\Rightarrow(6)$. Applying the action of $s_{-\vare_3}$ and Theorem \ref{thm::blocks}  if necessary, we may assume that $c=1$. If $d= \frac{k}{4}$ then we note that $$\la +a (\delta+\vare_1) = [{k}/{4}+a|~{k}/{4}+a,~{k}/{4}+a,~-2a -{k}/{2}] =[\ell|~\ell,~\ell,-2\ell] = \mu_{[\ell]},$$ where $\ell := {k}/{4}+a\in {1}/{4}+\frac{\Z}{2}$.

If $d=\frac{3a}{2}+\frac{k}{4}$, then we note that
$$ \la = \mu_{[\ell]} +a(\delta +\vare_2),$$
where $\ell :=\frac{k}{4}+\frac{a}{2}\in \frac{1}{4}+\frac{\Z}{2}$. This proves $(5)\Rightarrow (6)$.

Finally suppose that  $(\gamma, \alpha)$ is even. We may observe that  $(7)$ and $(8)$ are equivalent. It remains to show that they are equivalent to (9).

We first  note that $(9)\Rightarrow(8)$. Conversely, suppose (8) holds.  Applying the action of $s_{-\vare_3}$ and Theorem \ref{thm::blocks} if necessary, we may assume that $c=1$.  Then we have
$$ \la =[\ell|~{1}/{4},~-\ell -{1}/{4},~\ell] + ({1-k}/{2})[1|~-{1}/{2},~-{1}/{2},~1] =  \nu_{[\ell]} +({1-k}/{2})(\delta+\vare_3).
$$ where $\ell =-{3a+1}/{2}$. Since $-{3a+k}/{2}\notin \frac{1}{2}+\Z$, $a$ is odd and we have $\ell \in 1+3\Z$.  This proves $(8)\Rightarrow(9)$. This concludes the proof.
\end{proof}

\subsubsection{Reduction method} \label{subsect::Vblocksreduction}
In this section, we assume that $\la\in \h^\ast$ is an atypical weight such that $W_\la^1\cong V$. Recall the possibilities for $W_\la$ and $Z(\la)$ in
\eqref{eq::38}--\eqref{eq::310}.
Let $\alpha \in A(\la)$ and $\gamma \in Z(\la)$. We define  $$\la' = \begin{cases} s_2\la , &\mbox{ for  $\vare_2 \in Z(\la)$;}\\
s_2s_1 \la,&\mbox{ for $\vare_1\in Z(\la)$.}
\end{cases}$$
The following corollary reduces all cases to the case when $Z(\la) =\{-\vare_3, \vare_2-\vare_1\}$.

\begin{cor}   Let $\la\in\h^\ast$ be atypical with $W_\la^1\cong V$, $\alpha\in A(\la)$ and $\gamma\in Z(\la)$.
	\begin{itemize}
		\item[(1)] Suppose $(\gamma, \eta)$ is odd. Then $\mc O_\la$ is equivalent to the block in Theorem \ref{thm::51}(II) as a highest weight category.
		\item[(2)] Suppose $(\gamma, \eta)$ is even.
\begin{itemize}
\item[(i)] If $\langle\la,(2\delta)^\vee\rangle\not\in\hf+\Z$, then $\mc O_\la$ is equivalent to the block in Theorem \ref{thm::51}(III) as a highest weight category.
\item[(ii)] If $\langle\la,(2\delta)^\vee\rangle\in\hf+\Z$, then $\mc O_\la$ is equivalent to the block in Theorem \ref{thm::51}(I) as a highest weight category.
\end{itemize}
	\end{itemize}
\end{cor}

\begin{proof}
We first prove that $\mc O_\la\cong \mc O_{\la'}$ as highest weight categories.
	
	Assume that $\vare_2\in Z(\la)$. Then the good diagram $[A(\la)|\vare_2, \vare_{1}-\vare_3] \xrightarrow{s_2} [A(\la')|-\vare_3, \vare_2-\vare_1]$ implies that $T_{s_2}: \mc O_\la \rightarrow \mc O_{\la'}$ is an equivalence by Lemma \ref{lem::tequiv}.
	
	Assume that $\vare_1\in Z(\la)$. Again,  the composition of good diagrams $$[A(\la)|\vare_1, \vare_2-\vare_3] \xrightarrow{s_1} [A(s_1\la)|\vare_2, \vare_1-\vare_3]\xrightarrow{s_2} [A(\la')|-\vare_3, \vare_2-\vare_1],$$ implies that $T_{s_2}T_{s_1}: \mc O_\la \rightarrow \mc O_{\la'}$ is an equivalence by Lemma \ref{lem::tequiv}.

Thus, we have reduced all cases with $W_{\la'}\cong V$ to the cases in Theorem \ref{thm::51}. Now we make the simple observation that
\begin{align*}
(\gamma,\alpha)=(s_i\gamma,s_i\alpha), \quad i=1,2,
\end{align*}
and the corollary follows.
\end{proof}

\subsection{Tilting characters}

\subsubsection{Tilting characters in $\mc O_{\lambda_{[\ell]}}$}

Fix a number $k\in\mathbb{Z}$ such that $k\leq\ell$. We denote by ${}_\ell\la_{k}$ the anti-dominant weight in $W_{\la_{[\ell]}}([-\ell-\frac{1}{2}+k~|~\frac{1}{4}-\frac{k}{2},~\ell+\frac{1}{4}-\frac{k}{2},~ -\ell -\frac{1}{2}+k])$. Note that $_\ell\la_{k}=_\ell\la_{2\ell-k+1}$ by definition. Hence to calculate the characters below we only need to consider the case $k\le \ell$.

Set
${}_\ell\la_{k}^1=s_1({}_\ell\la_{k})$, ${}_\ell\la_{k}^2=s_{\epsilon_3}({}_\ell\la_{k})$ and ${}_\ell\la_{k}^{12}=s_1s_{\epsilon_3}({}_\ell\la_{k})$. Furthermore, set ${}_\ell\la_{k}^{\imath+}=s_0({}_\ell\la_{k}^\imath), \imath=\emptyset,1,2,12$.


All tilting characters $T_\lambda$ in the following Theorem \ref{thm:6.3} and Theorem \ref{thm:6.4} can be obtained by applying the translation functor, associated with tensoring with the adjoint module, on $T_{\lambda-2\delta}=M_{\lambda-2\delta}$.

\begin{thm}\label{thm:6.3}
Assume $\ell\in3\mathbb{Z}$ with $\ell\not=0$ and $\ell\geq k\in\mathbb{Z}$.
\begin{itemize}
  \item[(1)]
If $k<\ell$, then
\begin{align*}
T_{{}_\ell\la_{k}}=&M_{{}_\ell\la_{k}}+M_{{}_\ell\la_{k-1}},\\
T_{{}_\ell\la_{k}^1}=&M_{{}_\ell\la_{k}^1}+M_{{}_\ell\la_{k}}+M_{{}_\ell\la_{k-1}^1}+M_{{}_\ell\la_{k-1}},\\
T_{{}_\ell\la_{k}^2}=&M_{{}_\ell\la_{k}^2}+M_{{}_\ell\la_{k}}+M_{{}_\ell\la_{k-1}^2}+M_{{}_\ell\la_{k-1}},\\
T_{{}_\ell\la_{k}^{12}}=&M_{{}_\ell\la_{k}^{12}}+M_{{}_\ell\la_{k}^1}+M_{{}_\ell\la_{k}^{2}}+M_{{}_\ell\la_{k}}+M_{{}_\ell\la_{k-1}^{12}}+M_{{}_\ell\la_{k-1}^1}+M_{{}_\ell\la_{k-1}^2}+M_{{}_\ell\la_{k-1}},\\
T_{{}_\ell\la_{k}^+}=&M_{{}_\ell\la_{k}^+}+M_{{}_\ell\la_{k}}+M_{{}_\ell\la_{k+1}^+}+M_{{}_\ell\la_{k+1}},\\
T_{{}_\ell\la_{k}^{1+}}=&M_{{}_\ell\la_{k}^{1+}}+M_{{}_\ell\la_{k}^1}+M_{{}_\ell\la_{k}^+}+M_{{}_\ell\la_{k}}+M_{{}_\ell\la_{k+1}^{1+}}+M_{{}_\ell\la_{k+1}^1}+M_{{}_\ell\la_{k+1}^{+}}+M_{{}_\ell\la_{k+1}},\\
T_{{}_\ell\la_{k}^{2+}}=&M_{{}_\ell\la_{k}^{2+}}+M_{{}_\ell\la_{k}^2}+M_{{}_\ell\la_{k}^+}+M_{{}_\ell\la_{k}}+M_{{}_\ell\la_{k+1}^{2+}}+M_{{}_\ell\la_{k+1}^2}+M_{{}_\ell\la_{k+1}^{+}}+M_{{}_\ell\la_{k+1}},\\
T_{{}_\ell\la_{k}^{12+}}=&M_{{}_\ell\la_{k}^{12+}}+M_{{}_\ell\la_{k}^{12}}+M_{{}_\ell\la_{k}^{1+}}+M_{{}_\ell\la_{k}^{2+}}+M_{{}_\ell\la_{k}^{1}}+M_{{}_\ell\la_{k}^{2}}+M_{{}_\ell\la_{k}^{+}}+M_{{}_\ell\la_{k}}\\
&+M_{{}_\ell\la_{k+1}^{12+}}+M_{{}_\ell\la_{k+1}^{12}}+M_{{}_\ell\la_{k+1}^{1+}}+M_{{}_\ell\la_{k+1}^{2+}}+M_{{}_\ell\la_{k+1}^1}+M_{{}_\ell\la_{k+1}^2}+M_{{}_\ell\la_{k+1}^+}+M_{{}_\ell\la_{k+1}}.\\
\end{align*}
\item[(2)] If $k=\ell$, then
\begin{align*}
 T_{{}_\ell\la_{\ell}}=&M_{{}_\ell\la_{\ell}}+M_{{}_\ell\la_{\ell-1}},\\
 T_{{}_\ell\la_{\ell}^1}=&M_{{}_\ell\la_{\ell}^1}+M_{{}_\ell\la_{\ell}}+M_{{}_\ell\la_{\ell-1}^1}+M_{{}_\ell\la_{\ell-1}},\\
T_{{}_\ell\la_{\ell}^2}=&M_{{}_\ell\la_{\ell}^2}+M_{{}_\ell\la_{\ell}}+M_{{}_\ell\la_{\ell-1}^2}+M_{{}_\ell\la_{\ell-1}},\\
T_{{}_\ell\la_{\ell}^{12}}=&M_{{}_\ell\la_{\ell}^{12}}+M_{{}_\ell\la_{\ell}^1}+M_{{}_\ell\la_{\ell}^{2}}+M_{{}_\ell\la_{\ell}}+M_{{}_\ell\la_{\ell-1}^{12}}+M_{{}_\ell\la_{\ell-1}^1}+M_{{}_\ell\la_{\ell-1}^2}+M_{{}_\ell\la_{\ell-1}},\\
T_{{}_\ell\la_{\ell}^+}=&M_{{}_\ell\la_{\ell}^+}+M_{{}_\ell\la_{\ell}}+M_{{}_\ell\la_{\ell}^2}+M_{{}_\ell\la_{\ell-1}},\\
T_{{}_\ell\la_{\ell}^{1+}}=&M_{{}_\ell\la_{\ell}^{1+}}+M_{{}_\ell\la_{\ell}^1}+M_{{}_\ell\la_{\ell}^+}+M_{{}_\ell\la_{\ell}}+M_{{}_\ell\la_{\ell}^{12}}+M_{{}_\ell\la_{\ell}^2}+M_{{}_\ell\la_{\ell-1}^1}+M_{{}_\ell\la_{\ell-1}},\\
T_{{}_\ell\la_{\ell}^{2+}}=&M_{{}_\ell\la_{\ell}^{2+}}+M_{{}_\ell\la_{\ell}^2}+M_{{}_\ell\la_{\ell}^+}+M_{{}_\ell\la_{\ell}},\\
T_{{}_\ell\la_{\ell}^{12+}}=&M_{{}_\ell\la_{\ell}^{12+}}+M_{{}_\ell\la_{\ell}^{12}}+M_{{}_\ell\la_{\ell}^{1+}}+M_{{}_\ell\la_{\ell}^{2+}}+M_{{}_\ell\la_{\ell}^{1}}+M_{{}_\ell\la_{\ell}^{2}}+M_{{}_\ell\la_{\ell}^{+}}+M_{{}_\ell\la_{\ell}}.\\
\end{align*}
\end{itemize}
\end{thm}

For $\ell=0$ and $k\le 0$ we have ${}_0\la_{k}=[-\frac{1}{2}+k|-\frac{1}{4}+\frac{k}{2},-\frac{1}{4}+\frac{k}{2},\frac{1}{2}-k]$ so that ${}_0\la_{k}^1={}_0\la_{k}$.

\begin{thm} \label{thm:6.4}
Assume $\ell=0$.
  \begin{itemize}
    \item[(1)] If $k=0$, then
    \begin{align*}
      T_{{}_0\la_{0}}=&M_{{}_0\la_{0}}+M_{{}_0\la_{-1}},\\
      T_{{}_0\la_{0}^2}=&M_{{}_0\la_{0}^2}+M_{{}_0\la_{0}}+M_{{}_0\la_{-1}^2}+M_{{}_0\la_{-1}},\\ T_{{}_0\la_{0}^+}=&M_{{}_0\la_{0}^+}+M_{{}_0\la_{0}}+M_{{}_0\la_{0}^2}+M_{{}_0\la_{-1}},\\ T_{{}_0\la_{0}^{2+}}=&M_{{}_0\la_{0}^{2+}}+M_{{}_0\la_{0}^2}+M_{{}_0\la_{0}^+}+M_{{}_0\la_{0}}. \end{align*}
    \item[(2)] If $k<0$, then
    \begin{align*}
      T_{{}_0\la_{k}}=&M_{{}_0\la_{k}}+M_{{}_0\la_{k-1}},\\
      T_{{}_0\la_{k}^2}=&M_{{}_0\la_{k}^2}+M_{{}_0\la_{k}}+M_{{}_0\la_{k-1}^2}+M_{{}_0\la_{k-1}},\\ T_{{}_0\la_{k}^+}=&M_{{}_0\la_{k}^+}+M_{{}_0\la_{k}}+M_{{}_0\la_{k+1}^+}+M_{{}_0\la_{k+1}},\\ T_{{}_0\la_{k}^{2+}}=&M_{{}_0\la_{k}^{2+}}+M_{{}_0\la_{k}^+}+M_{{}_0\la_{k}^2}+M_{{}_0\la_{k}}+M_{{}_0\la_{k+1}^{2+}}+M_{{}_0\la_{k+1}^+}+M_{{}_0\la_{k+1}^2}+M_{{}_0\la_{k+1}}. \end{align*}
  \end{itemize}
\end{thm}

\subsubsection{Tilting characters in $\mc O_{\mu_{[\ell]}}$}

Assume $\ell\in\mathbb{Z}\pm\frac{1}{4}$. Fix $k\in\mathbb{Z}$.

We denote ${}_\ell\mu_{k}$ the unique anti-dominant weight in $W_{\mu_{[\ell]}}([\ell+k|\ell+k,\ell-\frac{1}{2}k,-2\ell-\frac{1}{2}k])$. Let
${}_\ell\mu_{k}^1=s_1({}_\ell\mu_{k})$, ${}_\ell\mu_{k}^2=s_{\epsilon_3}({}_\ell\mu_{k})$ and
${}_\ell\mu_{k}^{12}=s_1s_{\epsilon_3}({}_\ell\mu_{k})$. Note that ${}_\ell\mu_{k}={}_\ell\mu_{k}^1$ and ${}_\ell\mu_{k}^2={}_\ell\mu_{k}^{12}$ (respectively, ${}_\ell\mu_{k}={}_\ell\mu_{k}^2$ and ${}_\ell\mu_{k}^1={}_\ell\mu_{k}^{12}$) if and only if $k=0$ (respectively, $k=-4\ell$).

The translation functor associated with tensoring with the adjoint module on $T_{\lambda-2\delta}=M_{\lambda-2\delta}$ still works to obtain the following theorem.

\begin{thm}
\begin{itemize}
\item[(1)]
If $k\neq0,1,-4\ell,-4\ell+1$, then
\begin{align*}
    T_{{}_\ell\mu_{k}}=&M_{{}_\ell\mu_{k}}+M_{{}_{\ell}\mu_{k-1}},\\
    T_{{}_\ell\mu_{k}^1}=&M_{{}_\ell\mu_{k}^1}+M_{{}_\ell\mu_{k}}+M_{{}_{\ell}\mu_{k-1}^1}+M_{{}_{\ell}\mu_{k-1}},\\
    T_{{}_\ell\mu_{k}^2}=&M_{{}_\ell\mu_{k}^2}+M_{{}_{\ell}\mu_{k}}+M_{{}_\ell\mu_{k-1}^2}+M_{{}_{\ell}\mu_{k-1}},\\
    T_{{}_\ell\mu_{k}^{12}}=&M_{{}_\ell\mu_{k}^{12}}+M_{{}_{\ell}\mu_{k}^{1}}+M_{{}_\ell\mu_{k}^{2}}+M_{{}_{\ell}\mu_{k}}
    +M_{{}_\ell\mu_{k-1}^{12}}+M_{{}_{\ell}\mu_{k-1}^{1}}+M_{{}_\ell\mu_{k-1}^{2}}+M_{{}_{\ell}\mu_{k-1}}.
  \end{align*}
\item[(2)] If $k=0\neq-4\ell+1$, then
\begin{align*}
T_{{}_\ell\mu_{0}}=&M_{{}_\ell\mu_{0}}+M_{{}_{\ell}\mu_{-1}}+M_{{}_{\ell}\mu_{-1}^1},\\
       T_{{}_\ell\mu_{0}^2}=&M_{{}_\ell\mu_{0}^2}+M_{{}_{\ell}\mu_{0}}+M_{{}_\ell\mu_{-1}^2}+M_{{}_{\ell}\mu_{-1}}+M_{{}_\ell\mu_{-1}^{12}}+M_{{}_{\ell}\mu_{-1}^1}.\\
    \end{align*}
\item[(3)] If $k=1\neq-4\ell$, then
\begin{align*}
    T_{{}_\ell\mu_{1}}=&M_{{}_\ell\mu_{1}}+M_{{}_{\ell}\mu_{0}},\\
    T_{{}_\ell\mu_{1}^1}=&M_{{}_\ell\mu_{1}^1}+M_{{}_\ell\mu_{1}}+M_{{}_{\ell}\mu_{0}},\\
    T_{{}_\ell\mu_{1}^2}=&M_{{}_\ell\mu_{1}^2}+M_{{}_{\ell}\mu_{1}}+M_{{}_\ell\mu_{0}^2}+M_{{}_{\ell}\mu_{0}},\\
    T_{{}_\ell\mu_{1}^{12}}=&M_{{}_\ell\mu_{1}^{12}}+M_{{}_{\ell}\mu_{1}^{1}}+M_{{}_\ell\mu_{1}^{2}}+M_{{}_{\ell}\mu_{1}}
    +M_{{}_\ell\mu_{0}^{2}}+M_{{}_{\ell}\mu_{0}}.
  \end{align*}
  \item[(4)] If $k=-4\ell\neq1$, then
  \begin{align*}
    T_{{}_\ell\mu_{-4\ell}}=&M_{{}_\ell\mu_{-4\ell}}+M_{{}_{\ell}\mu_{-4\ell-1}}+M_{{}_\ell\mu_{-4\ell-1}^2},\\
    T_{{}_\ell\mu_{-4\ell}^1}=&M_{{}_\ell\mu_{-4\ell}^1}+M_{{}_\ell\mu_{-4\ell}}+M_{{}_\ell\mu_{-4\ell-1}^{12}}+M_{{}_{\ell}\mu_{-4\ell-1}^{1}}+M_{{}_\ell\mu_{-4\ell-1}^{2}}+M_{{}_{\ell}\mu_{-4\ell-1}}.
  \end{align*}
  \item[(5)] If $k=-4\ell+1\neq0$, then
  \begin{align*}
    T_{{}_\ell\mu_{k}}=&M_{{}_\ell\mu_{k}}+M_{{}_{\ell}\mu_{k-1}},\\
    T_{{}_\ell\mu_{k}^1}=&M_{{}_\ell\mu_{k}^1}+M_{{}_\ell\mu_{k}}+M_{{}_{\ell}\mu_{k-1}^1}+M_{{}_{\ell}\mu_{k-1}},\\
    T_{{}_\ell\mu_{k}^2}=&M_{{}_\ell\mu_{k}^2}+M_{{}_{\ell}\mu_{k}}+M_{{}_{\ell}\mu_{k-1}},\\
    T_{{}_\ell\mu_{k}^{12}}=&M_{{}_\ell\mu_{k}^{12}}+M_{{}_{\ell}\mu_{k}^{1}}+M_{{}_\ell\mu_{k}^{2}}+M_{{}_{\ell}\mu_{k}}
    +M_{{}_{\ell}\mu_{k-1}^{1}}+M_{{}_{\ell}\mu_{k-1}}.
  \end{align*}
  \item[(6)] If $k=0=-4\ell+1$, then
  \begin{align*}
    T_{_{\frac{1}{4}}\mu_0}=&M_{_{\frac{1}{4}}\mu_0}+M_{_{\frac{1}{4}}\mu_{-1}^1}+M_{_{\frac{1}{4}}\mu_{-1}}
    +M_{_{\frac{1}{4}}\mu_{-2}^1}+M_{_{\frac{1}{4}}\mu_{-2}},\\
     T_{_{\frac{1}{4}}\mu^2_0}=&M_{_{\frac{1}{4}}\mu^2_0}+M_{_{\frac{1}{4}}\mu_0}+M_{_{\frac{1}{4}}\mu^1_{-1}}
     +M_{_{\frac{1}{4}}\mu_{-1}}.
  \end{align*}
\item[(7)] If $k=1=-4\ell$, then
\begin{align*}
    T_{_{-\frac{1}{4}}\mu_1}=&M_{_{-\frac{1}{4}}\mu_1}+M_{_{-\frac{1}{4}}\mu^2_0}
    +M_{_{-\frac{1}{4}}\mu_0},\\
    T_{_{-\frac{1}{4}}\mu^1_1}=&M_{_{-\frac{1}{4}}\mu^1_1}+M_{_{-\frac{1}{4}}\mu_1}
    +M_{_{-\frac{1}{4}}\mu^2_0}+M_{_{-\frac{1}{4}}\mu_0}.
\end{align*}

\end{itemize}
\end{thm}

\subsubsection{Tilting characters in $\mc O_{\nu_{[\ell]}}$}

Assume $\ell\in3\mathbb{Z}+1$ and fix $k\in\mathbb{Z}$. Denote ${}_\ell\nu_k$ the unique anti-dominant weight in $W_{\nu_{[\ell]}}([\ell+k|\frac{1}{4}-\frac{k}{2},-\ell-\frac{1}{4}-\frac{k}{2},\ell+k])$.
Let ${}_\ell\nu_k^1=s_1({}_\ell\nu_k)$, ${}_\ell\nu_k^2=s_{\epsilon_3}({}_\ell\nu_k)$ and ${}_\ell\nu_k^{12}=s_1s_{\epsilon_3}({}_\ell\nu_k)$. Note that ${}_\ell\nu_k^2={}_\ell\nu_k$ and ${}_\ell\nu_k^{12}={}_\ell\nu_k^1$ if and only if $\ell+k=0$.

\begin{thm}
  Assume $\ell\in3\mathbb{Z}+1$ and $k\in \mathbb{Z}$.
  \begin{itemize}
    \item[(1)] If $k\neq-\ell, -\ell+1$, then
     \begin{align*}
    T_{{}_\ell\nu_{k}}=&M_{{}_\ell\nu_{k}}+M_{{}_{\ell}\nu_{k-1}},\\
    T_{{}_\ell\nu_{k}^1}=&M_{{}_\ell\nu_{k}^1}+M_{{}_\ell\nu_{k}}+M_{{}_{\ell}\nu_{k-1}^1}+M_{{}_{\ell}\nu_{k-1}},\\
    T_{{}_\ell\nu_{k}^2}=&M_{{}_\ell\nu_{k}^2}+M_{{}_{\ell}\nu_{k}}+M_{{}_{\ell}\nu_{k-1}^2}+M_{{}_{\ell}\nu_{k-1}},\\
    T_{{}_\ell\nu_{k}^{12}}=&M_{{}_\ell\nu_{k}^{12}}+M_{{}_{\ell}\nu_{k}^{1}}+M_{{}_\ell\nu_{k}^{2}}+M_{{}_{\ell}\nu_{k}}
    +M_{{}_{\ell}\nu_{k-1}^{12}}+M_{{}_{\ell}\nu_{k-1}^1}+M_{{}_{\ell}\nu_{k-1}^{2}}+M_{{}_{\ell}\nu_{k-1}}.
  \end{align*}
  \item[(2)] If $k=-\ell$, then
  \begin{align*}
    T_{{}_\ell\nu_{-\ell}}=&M_{{}_\ell\nu_{-\ell}}+M_{{}_{\ell}\nu_{-\ell-1}}+M_{{}_{\ell}\nu_{-\ell-1}^2},\\
    T_{{}_\ell\nu_{-\ell}^1}=&M_{{}_\ell\nu_{-\ell}^1}+M_{{}_\ell\nu_{-\ell}}+M_{{}_{\ell}\nu_{-\ell-1}^1}+M_{{}_{\ell}\nu_{-\ell-1}^{12}}+M_{{}_{\ell}\nu_{-\ell-1}}+M_{{}_{\ell}\nu_{-\ell-1}^{2}}.
    \end{align*}
  \item[(3)] If $k=-\ell+1$, then
  \begin{align*}
    T_{{}_\ell\nu_{-\ell+1}}=&M_{{}_\ell\nu_{-\ell+1}}+M_{{}_{\ell}\nu_{-\ell}}+M_{{}_\ell\nu_{-\ell-1}},\\
    T_{{}_\ell\nu_{-\ell+1}^1}=&M_{{}_\ell\nu_{-\ell+1}^1}+M_{{}_\ell\nu_{-\ell+1}}+M_{{}_{\ell}\nu_{-\ell}^1}+M_{{}_{\ell}\nu_{-\ell}} +M_{{}_\ell\nu_{-\ell-1}^1}+M_{{}_\ell\nu_{-\ell-1}},\\
    T_{{}_\ell\nu_{-\ell+1}^2}=&M_{{}_\ell\nu_{-\ell+1}^2}+M_{{}_{\ell}\nu_{-\ell+1}}+M_{{}_{\ell}\nu_{-\ell}},\\
    T_{{}_\ell\nu_{-\ell+1}^{12}}=&M_{{}_\ell\nu_{-\ell+1}^{12}}+M_{{}_{\ell}\nu_{-\ell+1}^{1}}+M_{{}_\ell\nu_{-\ell+1}^{2}}+M_{{}_{\ell}\nu_{-\ell+1}}
    +M_{{}_{\ell}\nu_{-\ell}^1}+M_{{}_{\ell}\nu_{-\ell}}.
  \end{align*}
  \end{itemize}
\end{thm}
\begin{proof}
All formulas except the cases of $k=-\ell+2, -\ell+1$, can be obtained by applying the translation functor, associated with tensoring with the adjoint module, on $T_{\lambda-2\delta}=M_{\lambda-2\delta}$.

For $k=-\ell+2,-\ell+1$, we apply the translation functor to $T_{\lambda-(\delta+\epsilon_2)}=M_{\lambda-(\delta+\epsilon_2)}$.
\end{proof}

\section{Character formulas in $\mc O_{S_3}$ cases} \label{sect::OS3case}

\subsection{Blocks in   $\mc O_{S_3}$ cases}

 In this section, we assume that $\la = d\delta +a\omega_1 +b\omega_2$ is an atypical weight such that $W_\la^1 \cong S_3$.

 For each $\ell\in \N$ with $\ell\not \equiv 0~(\text{mod }3)$, we define the following  weights  $$\la_{[\ell]}:=\left[-{\ell}/{2}~\middle|~- {\ell}/{2}, 0, {\ell}/{2}\right],$$
  $$\la'_{[\ell]}:=\left[-{\ell}/{2}~\middle|~ 0, -{\ell}/{2}, {\ell}/{2}\right].$$

Note that $S_3\cong\langle s_2=s_{\ep_1},s_{\ep_2},s_{-\ep_3}\rangle$ and observe that
 $$W_{\la_{[\ell]}} = \begin{cases}
 \Z_2\times S_3,& \text{ if $\ell$ is odd}, \\
 S_3,& \text{ if $\ell$ is even}.
 \end{cases}$$


\begin{lem} \label{lem::61}
For each $\ell\in \N$, we have  $[{\ell}/{2}|- {\ell}/{2}, 0,{\ell}/{2}]\in\emph{Irr}\mc O_{\la_{[\ell]}}$ and $[{\ell}/{2}|0,- {\ell}/{2},{\ell}/{2}]\in\emph{Irr}\mc O_{\la'_{[\ell]}}$.
\end{lem}
\begin{proof}
Since $(\la_{[\ell]},\alpha)=0$, $\alpha=\delta+\vare_1$, we need to show that $[{\ell}/{2}|- {\ell}/{2}, 0,{\ell}/{2}]\in W_{\la_{[\ell]}}(\la_{[\ell]}+\Z\alpha)$.

Indeed, we have $[{\ell}/{2}|- {\ell}/{2}, 0,{\ell}/{2}] =  s_2(\la_{[\ell]}+\ell\alpha)$.

Similarly, we have $(\la'_{[\ell]},\beta)=0$ with $\beta=\delta+\vare_2$, and so we need to show that $[{\ell}/{2}|0,- {\ell}/{2}, {\ell}/{2}]\in W_{\la'_{[\ell]}}(\la'_{[\ell]}+\Z\beta)$. Here we observe that $[{\ell}/{2}|0,- {\ell}/{2},{\ell}/{2}] =  s_{\epsilon_2}(\la'_{[\ell]}+\ell\beta)$.
\end{proof}
	
\begin{prop} Let $\la\in\h^\ast$ be atypical with $W_\la^1\cong S_3$.
\begin{itemize}
	\item[(1)] Then there exists $\ell\in \N$ such that either $\la \in \emph{Irr}\mc O_{\la_{[\ell]}}$ or  $\la \in \emph{Irr}\mc O_{\la'_{[\ell]}}$.
	\item[(2)] The sets $\emph{Irr}\mc O_{\la_{[\ell]}},  \emph{Irr}\mc O_{\la_{[k]}},$ and $\emph{Irr}\mc O_{\la'_{[k]}}$ are distinct for all $\ell\neq k$.
	\item[(3)] The twisting functor $T_{s_1}: \mc O_{\la_{[\ell]}}\rightarrow \mc O_{\la'_{[\ell]}}$ is an equivalence of highest weight categories.
\end{itemize}
\end{prop}
\begin{proof}
	We first prove $(1).$ It is shown in the proof of Theorem \ref{thm::desnonblocks} that if $\eta$ is   a short root and $\gamma $ is  a long root, then $s_{\eta}\in W_\la$ implies  $s_\gamma \notin W_\la$.  By   \eqref{eq::Weylgpcoroot} we have $b\in \Z$, $3a+b\in \Z$, $3a+2b\in \Z$, $a\notin \Z$, $a+b\notin \Z$ and $2a+b\notin \Z$. This implies that $a\notin \frac{\Z}{2}$ and $a\in \frac{\Z}{3}$. We set $k \in \Z$ such that $k \not \equiv 0 (\text{ mod $3$})$. Now, we have $$\la = \left[d~\middle|~ {b}/{2}, ~(k+b)/2,~-{k}/{2}-b \right].$$
	
	Since $\la$ is atypical, we have the following possibilities $d=\pm \frac{b}{2},~\pm(\frac{k}{2} +\frac{b}{2}),~\pm(\frac{k}{2}+b)$.

    Using the Weyl group $W_\la^1$ if necessary, we may assume that $d= \frac{b}{2}$, $\frac{k+b}{2}$, or $-\frac{k}{2}-b$.
    We will proceed with the proof case by case.

\begin{enumerate}
	\item[(a)] Suppose that $d=\frac{b}{2}.$ In this case we have \begin{align*}
	&\la-(2b+k)[1|1,~-{1}/{2},~-{1}/{2}] =[-{3b}/{2}-k|~-{3b}/{2}-k,~ {3b}/{2}+k,~0]=[{\ell}/{2}|~{\ell}/{2},~- {\ell}/{2},~0].
	\end{align*} where $\ell=-3b-2k$. Observe that $\ell\not \equiv 0 (\text{mod }3)$. If $\ell >0$ then
	$$s_2[{\ell}/{2}|~{\ell}/{2},~- {\ell}/{2},~0] = [{\ell}/{2}|~-{\ell}/{2},0~,~ {\ell}/{2}],$$
	implies that $\la \in \text{Irr}\mc O_{\la_{[\ell]}}$ by Lemma \ref{lem::61}. If $\ell <0$ then $$s_{\vare_2}[{\ell}/{2}|~{\ell}/{2},~- {\ell}/{2},~0] = [{\ell}/{2}|~0,{\ell}/{2},~ -{\ell}/{2}] =\la'_{[-\ell]}.$$
	\item[(b)]
	Suppose that $d= (k+b)/{2}$. Then
	\begin{align*}
	&\la +b[1|~- {1}/{2},1, - {1}/{2}] = [\ell| ~0, \ell,~-\ell],
	\end{align*} where $\ell=  {k}/{2}+ {3b}/{2}$ with $\ell\not \equiv 0(\text{ mod }3)$. If $\ell <0$ then  $[\ell| ~0, \ell,~-\ell] = \la'_{-2\ell}.$ If $\ell>0$ then $$s_{-\vare_3}[\ell| ~0, \ell,~-\ell] = [\ell| -\ell, 0,~\ell].$$ By Lemma \ref{lem::61}, $\la \in \text{Irr}\mc O_{\la_{2\ell}}.$
	\item[(c)]
	Supoose that $d=-\frac{k}{2}-b$. Then \begin{align*}
	&\la+(k +b)[1|- {1}/{2},~- {1}/{2},~1] =[ {k}/{2}|- {k}/{2},~0, {k}/{2}].
	\end{align*} If $k>0$ then the weight above is linked to $\la_{[k]}$ by Lemma \ref{lem::61}.  If $k<0$, then we note that
	 \[s_{-\vare_3}[ {k}/{2}|- {k}/{2},~0, {k}/{2}] = [ {k}/{2}|~0,~{k}/{2}, -{k}/{2}]= \la'_{[-k]}. \] By Lemma \ref{lem::61}, $\la \in \text{Irr}\mc O_{\la_{[-k]}}.$
\end{enumerate}
	This completes the proof of the Part $(1)$.

Part (2) is a consequence of Theorem \ref{thm::blocks}, while Part (3) is a consequence of Lemma \ref{lem::tequiv}.
\end{proof}

\subsection{Tilting characters}

\subsubsection{Case of $\ell$ even}
Let $\ell\in2\mathbb{N}\setminus 3\mathbb{N}$ and $k\in\mathbb{Z}$. Now $W_{\la_{[\ell]}}=S_3=\langle s_{\epsilon_1},s_{\epsilon_2}\rangle$. Denote by
${}_\ell\la_k$ the unique antidominant weight in $S_3([-\frac{\ell}{2}+k|-\frac{\ell}{2}+k,-\frac{k}{2},\frac{\ell}{2}-\frac{k}{2}])$.
Let
${}_\ell\la_k^w=w({}_\ell\la_k)$ for any $w\in S_3$.

For any $k\in \mathbb{Z}$, let $Q_k$ be the set of minimal length representatives of the left coset of $Q^{(k)}$ in $S_3$, where $Q^{(k)}=\{w\in S_3~|~w({}_\ell\la_k)={}_\ell\la_k\}$.

{Denote
\begin{equation}\label{eq:H}
  H_1=\{e,s_{\epsilon_1},s_{\epsilon_2}s_{\epsilon_1}\},\quad H_2=\{e,s_{\epsilon_2},s_{\epsilon_1}s_{\epsilon_2}\}.
\end{equation}
In the following theorem we use the Bruhat order ``$\leq$'' on $S_3$ by regarding $s_{\epsilon_1}$ and $s_{\epsilon_2}$ as simple reflections.}
The formulas below can be obtained by applying the translation functor, associated with tensoring with the adjoint module, to $T_{\lambda-2\delta}$, $T_{\lambda-(\delta-\epsilon_3)}$ or $T_{\lambda-(\delta+\epsilon_2)}$.

\begin{thm} Let $\ell\in2\mathbb{N}\setminus 3\mathbb{N}$ and $k\in\mathbb{Z}$.
\begin{itemize}
  \item[(1)] If $k\neq0,1,\frac{\ell}{2},\frac{\ell}{2}+1,\ell,\ell+1$, then $Q_k=S_3$. For any $w\in S_3$,
  \begin{equation*}
    T_{{}_\ell\la_k^w}=\sum_{\sigma\leq w}(M_{{}_\ell\la_k^{\sigma}}+M_{{}_\ell\la_{k-1}^{\sigma}}).
  \end{equation*}
  \item[(2)] If $k=0$ or $k=\ell\neq2$, then $Q_k=H_1$. For any $w\in H_1$,
      \begin{equation*}
    T_{{}_\ell\la_k^w}=\sum_{\sigma\in H_1,\sigma\leq w}M_{{}_\ell\la_k^{\sigma}}+\sum_{\sigma\leq ws_{\epsilon_2}}M_{{}_\ell\la_{k-1}^{\sigma}}.
  \end{equation*}
  \item[(3)] If $k=\frac{\ell}{2}$ with $\ell\neq2$, then $Q_k=H_2$. For any $w\in Q_k=H_2$,
      \begin{equation*}
    T_{{}_\ell\la_k^w}=\sum_{\sigma\in H_2, \sigma\leq w} M_{{}_\ell\la_k^{\sigma}}+\sum_{\sigma \leq ws_{\epsilon_1}}M_{{}_\ell\la_{k-1}^{\sigma}}.
  \end{equation*}
  \item[(4)] Assume $k=1$ with $\ell\neq2$, or, $k=\ell+1$. In this case, $Q_k=S_3$. If $w\in H_1$, then
  $$ T_{{}_\ell\la_k^w}=\sum_{\sigma\leq w}M_{{}_\ell\la_k^{\sigma}}+\sum_{\sigma\in H_1,\sigma\leq w}M_{{}_\ell\la_{k-1}^{\sigma}}+\sum_{\sigma\leq w}M_{{}_\ell\la_{k-2}^{\sigma}};$$
  If $w\in S_3\backslash H_1$, then
  $$ T_{{}_\ell\la_k^w}=\sum_{\sigma\leq w}M_{{}_\ell\la_k^{\sigma}}+\sum_{\sigma\in H_1,\sigma\leq w}M_{{}_\ell\la_{k-1}^{\sigma}}.$$

  \item[(5)] Assume $k=\frac{\ell}{2}+1$ with $\ell\neq2$. In this case, $Q_k=S_3$. If $w\in H_2$, then
  $$ T_{{}_\ell\la_k^w}=\sum_{\sigma\leq w}M_{{}_\ell\la_k^{\sigma}}+\sum_{\sigma\in H_2,\sigma\leq w}M_{{}_\ell\la_{k-1}^{\sigma}}+\sum_{\sigma\leq w}M_{{}_\ell\la_{k-2}^{\sigma}};$$
  If $w\in S_3\backslash H_2$, then
  $$ T_{{}_\ell\la_k^w}=\sum_{\sigma\leq w}M_{{}_\ell\la_k^{\sigma}}+\sum_{\sigma\in H_2,\sigma\leq w}M_{{}_\ell\la_{k-1}^{\sigma}}.$$

    \item[(6)] If $\ell=2$ and $k=1$, then $Q_k=H_2$. We have
     \begin{align*}
    T_{_2\lambda_1}=&M_{_2\lambda_1}+M_{_2\lambda_0^{s_{\ep_1}}}+M_{_2\lambda_0}+M_{_2\lambda_{-1}^{s_{\ep_1}}}+M_{_2\lambda_{-1}};\\
    T_{_2\lambda_1^{s_{\ep_2}}}=&M_{_2\lambda_1^{s_{\ep_2}}}+M_{_2\lambda_1}+M_{_2\lambda_0^{s_{\ep_2}s_{\ep_1}}}+M_{_2\lambda_0^{s_{\ep_1}}}+M_{_2\lambda_0}\\
    &+M_{_2\lambda_{-1}^{s_{\ep_2}s_{\ep_1}}}+M_{_2\lambda_{-1}^{s_{\ep_2}}}+M_{_2\lambda_{-1}^{s_{\ep_1}}}+M_{_2\lambda_{-1}};\\
    T_{_2\lambda_1^{s_{\ep_1}s_{\ep_2}}}=&M_{_2\lambda_1^{s_{\ep_1}s_{\ep_2}}}+M_{_2\lambda_1^{s_{\ep_2}}}+M_{_2\lambda_1}+M_{_2\lambda_0^{s_{\ep_2}s_{\ep_1}}}+M_{_2\lambda_0^{s_{\ep_1}}}+M_{_2\lambda_0}.
  \end{align*}

  \item[(7)] If $\ell=k=2$, then $Q_k=H_1$. We have
  \begin{align*}
    T_{_2\lambda_2}&=M_{_2\lambda_2}+M_{_2\lambda_1^{s_{\ep_2}}}+M_{_2\lambda_1}+M_{_2\lambda_0};\\
    T_{_2\lambda_2^{s_{\ep_1}}}&=M_{_2\lambda_2^{s_{\ep_1}}}+M_{_2\lambda_2}+M_{_2\lambda_1^{s_{\ep_1}s_{\ep_2}}}+M_{_2\lambda_1^{s_{\ep_2}}}+M_{_2\lambda_1}+M_{_2\lambda_0^{s_{\ep_1}}}+M_{_2\lambda_0};\\
    T_{_2\lambda_2^{s_{\ep_2}s_{\ep_1}}}&=M_{_2\lambda_2^{s_{\ep_2}s_{\ep_1}}}+M_{_2\lambda_2^{s_{\ep_1}}}+M_{_2\lambda_2}+M_{_2\lambda_1^{s_{\ep_1}s_{\ep_2}}}+M_{_2\lambda_1^{s_{\ep_2}}}+M_{_2\lambda_1}.
  \end{align*}
\end{itemize}
   \end{thm}

\subsubsection{Case of $\ell$ odd}
Let $\ell\in2\mathbb{N}+1\setminus 3\mathbb{N}$. Now $W_{\la_{[\ell]}}=\mathbb{Z}_2\times S_3=\langle s_{0}, s_{\epsilon_1},s_{\epsilon_2}\rangle$. Denote by
${}_\ell\la_k$ the unique antidominant weight in $(\mathbb{Z}_2\times S_3)([-\frac{\ell}{2}+k|-\frac{\ell}{2}+k,-\frac{k}{2},\frac{\ell}{2}-\frac{k}{2}])$ and let
${}_\ell\la_k^w=w({}_\ell\la_k)$ for any $w\in \mathbb{Z}_2\times S_3$. Since $_\ell\la_k= {_\ell\la_{\ell-k}}$, for the computation of characters below it is enough to consider the case when $\frac{\ell}{2}>k\in\Z$.

For any $\frac{\ell}{2}>k\in\Z$, let $Q_k$ be the set of minimal length representatives of the left coset of $Q^{(k)}$ in $\mathbb{Z}_2\times S_3$, where $Q^{(k)}=\{w\in \mathbb{Z}_2\times S_3~|~w({}_\ell\la_k)={}_\ell\la_k\}$.

{Recall the sets $H_1$ and $H_2$ in \eqref{eq:H}. Let $$J=s_0S_3,\quad J_1=s_0H_1=\{s_0,s_0s_{\epsilon_1},s_0s_{\epsilon_2}s_{\epsilon_1}\}, \quad J_2=s_0H_2=\{s_0,s_0s_{\epsilon_2},s_0s_{\epsilon_1}s_{\epsilon_2}\}.$$
In the following theorem, we use the Bruhat order ``$\leq$'' on $\mathbb{Z}_2\times S_3$ by regarding $s_0$, $s_{\ep_1}$ and $s_{\ep_2}$ as simple reflections.}
The formulas below can be obtained by applying the translation functor, associated with tensoring with the adjoint module, to $T_{\lambda-2\delta}$.

\begin{thm}
Let $\ell\in2\mathbb{N}+1\setminus 3\mathbb{N}$, $\frac{\ell}{2}>k\in\mathbb{Z}$ and $w\in Q_k$.
\begin{itemize}
\item[(1)] Assume $k\neq0,\pm1,\frac{\ell-1}{2}$. In this case, $Q_k=\mathbb{Z}_2\times S_3$. If $w\in S_3$, then
\begin{equation}\label{eq:1}
  T_{_\ell\lambda_k^w}=\sum_{\sigma\leq w}(M_{_\ell\lambda_k^\sigma}+M_{_\ell\lambda_{k-1}^\sigma});
\end{equation}
if $w\in s_0S_3$, then
\begin{equation}\label{eq:2}
  T_{_\ell\lambda_k^w}=\sum_{\sigma\leq w}(M_{_\ell\lambda_k^\sigma}+M_{_\ell\lambda_{k+1}^\sigma}).
\end{equation}

\item[(2)] Assume $k=\frac{\ell-1}{2}$ with $\ell\neq1$. In this case, $Q_k=\mathbb{Z}_2\times S_3$. If $w\in S_3$, then the character of $T_{_\ell\lambda_k^w}$ is given by the same formula as the one in \eqref{eq:1}. If $s_0w\in J_2$ (where $w\in H_2$), then
    $$T_{{}_\ell\la_{\frac{\ell-1}{2}}^{s_0w}}=\sum_{\sigma\leq s_0w~\mbox{\tiny or}~ws_{\ep_1}}T_{{}_\ell\la_{\frac{\ell-1}{2}}^{\sigma}}+\sum_{\sigma\leq w}T_{{}_\ell\la_{\frac{\ell-3}{2}}^{\sigma}}.$$
    If $w\in s_0S_3\setminus J_2$, then
$$T_{{}_\ell\la_k^w}=\sum_{\sigma \leq w}M_{{}_\ell\la_k^{\sigma}}.$$

\item[(3)] Assume $k=1$. In this case, $Q_k=\mathbb{Z}_2\times S_3$. If $w\in H_1$, then
$$T_{{}_\ell\la_1^w}=\sum_{\sigma\leq w}M_{{}_\ell\la_1^\sigma}+\sum_{\sigma\in H_1, \sigma\leq w}M_{{}_\ell\la_0^\sigma}+\sum_{\sigma\leq w}M_{{}_\ell\la_{-1}^\sigma}.$$
If $w\in S_3\setminus H_1$, then
$$T_{{}_\ell\la_1^w}=\sum_{\sigma\leq w}M_{{}_\ell\la_1^\sigma}+\sum_{\sigma\in H_1, \sigma\leq w}M_{{}_\ell\la_0^\sigma}.$$
If $w\in s_0S_3$, then the character of $T_{_\ell\lambda_k^w}$ is given by the same formula as the one in \eqref{eq:2}.

\item[(4)] Assume $k=-1$. In this case, $Q_k=\mathbb{Z}_2\times S_3$. If $w\in J_1$, then
$$T_{{}_\ell\la_{-1}^w}=\sum_{\sigma\leq w}M_{{}_\ell\la_1^\sigma}+\sum_{\sigma\in J_1\sqcup H_1, \sigma\leq w}M_{{}_\ell\la_0^\sigma}+\sum_{\sigma\leq w}M_{{}_\ell\la_{1}^\sigma}.$$
If $w\in s_0S_3\setminus J_1$, then $$T_{{}_\ell\la_1^w}=\sum_{\sigma\leq w}M_{{}_\ell\la_1^\sigma}+\sum_{\sigma\in J_1\sqcup H_1, \sigma\leq w}M_{{}_\ell\la_0^\sigma}.$$
If $w\in S_3$, then the character of $T_{_\ell\lambda_k^w}$ is given by the same formula as the one in \eqref{eq:1}.

\item[(5)] Assume $k=0$ with $\ell\neq1$. In this case, $Q_k=J_1\sqcup H_1$. If $w\in H_1$, then
\begin{equation}\label{eq:3}
T_{{}_\ell\la_k^w}=\sum_{\sigma\in H_1, \sigma \leq w}M_{{}_\ell\la_k^{\sigma}}+\sum_{\sigma \leq ws_{\epsilon_2}}M_{{}_\ell\la_{k-1}^{\sigma}}.
\end{equation}
If $w\in J_1$, then
$$T_{{}_\ell\la_k^w}=\sum_{\sigma\in J_1\sqcup H_1, \sigma \leq w}M_{{}_\ell\la_k^{\sigma}}+\sum_{\sigma \leq ws_{\epsilon_2}}M_{{}_\ell\la_{k+1}^{\sigma}}.$$

\item[(6)] Assume $k=0$ and $\ell=1$. In this case, $Q_k=J_1\sqcup H_1$. If $w\in J_1$, then \begin{align*}
       T_{_{1}\la_0^{s_0}}=&M_{_{1}\la_0^{s_0}}+M_{_{1}\la_0}
       +M_{_{1}\la_0^{s_{\epsilon_1}}}+M_{_{1}\la_0^{s_{\epsilon_2}s_{\epsilon_1}}}
       +M_{_{1}\la_{-1}}+M_{_{1}\la_{-1}^{s_{\epsilon_1}}}+M_{_{1}\la_{-1}^{s_{\epsilon_2}}}+M_{_{1}\la_{-1}^{s_{\epsilon_2}s_{\epsilon_1}}},\\
      T_{_{1}\la_{0}^{s_0s_{\epsilon_1}}}=&M_{_{1}\la_{0}^{s_0s_{\epsilon_1}}}+M_{_{1}\la_{0}^{s_0}}+
      M_{_{1}\la_{0}^{s_{\epsilon_1}}}+M_{_{1}\la_{0}}+M_{_{1}\la_{0}^{s_{\epsilon_2}s_{\epsilon_1}}}
      +M_{_{1}\la_{-1}}+M_{_{1}\la_{-1}^{s_{\epsilon_1}}},\\
       T_{_{1}\la_0^{s_0s_{\epsilon_2}s_{\epsilon_1}}}=&M_{_{1}\la_0^{s_0s_{\epsilon_2}s_{\epsilon_1}}}+M_{_{1}\la_0^{s_0s_{\epsilon_1}}}+M_{_{1}\la_0^{s_0}}
       +M_{_{1}\la_0^{s_{\epsilon_2}s_{\epsilon_1}}}+M_{_{1}\la_0^{s_{\epsilon_1}}}+M_{_{1}\la_0}.
        \end{align*}
        If $w\in H_1$, then the character of $T_{_\ell\lambda_k^w}$ is given by the same formula as the one in \eqref{eq:3}.
 \end{itemize}
\end{thm}

\section{Character formulas in $\mc O_{W_{G_2}}$} \label{sect::OWG2case}

\subsection{Blocks in $\mc O_{W_{G_2}}$ cases}
 We describe in this section the block decomposition of $\mc O_{W_{G_2}}$.  For each $a\in \Z$, we recall that $$a\omega_1 =a (\vare_1+2\vare_2) =\left[0|~0,~3a/2,~ -3a/2\right].$$

\begin{prop}
  We have
  \begin{align}
  &\mc O_{W_{G_2}} =\bigoplus_{a\in \N} \mc O_{a\omega_1}.
  \end{align}
\end{prop}
\begin{proof}
	We first note that $W_{a\omega_1} =W_{G_2}$ by \eqref{eq::Weylgpcoroot}. Next, let $$\la = D\delta +A\omega_1 +B\omega_2$$ be an atypical weight with  $W_\la^1 = W_{G_2}$. We claim that
	there exists $a\in \N$ such that $\mc O_\la =\mc O_{a\omega_1}$. To see this, we first assume that $(\la,\delta+c\vare_1)$, for $c=\pm 1$. This implies that $D=\pm B/2$. Since $\la$ is non-integral we have by \eqref{eq::Weylgpcoroot} that $D\in \Z$ and so $a:= A+B -cD \in \Z$. It follows that the weight
	\[\la-D(\delta+c\vare_1) = a\omega_1,\] lies in the set $\text{Irr}\mc O_\la$ by Theorem \ref{thm::blocks}.

Suppose $(\la, \delta+c\vare_i) =0$ with $i\neq 1$ for $c=\pm 1$. Since $W_\la = W_{G_2}$, there is $w\in W_\la$ such that $(w\la, \delta+c\vare_1)=0$. This case reduces to the previous case since $\mc O_\la =\mc O_{w\la}$ by Theorem \ref{thm::blocks}.
	
	Finally, we show that for any $a, b\in \Z$,
	  \[\mc O_{a\omega_1}= \mc O_{b\omega_1} \Leftrightarrow  a =\pm b.\] By a direct computation we have
	  \[W_{G_2}\omega_1 = \left\{\pm [0|3/2,-3/2,0], ~\pm [0|0,3/2,-3/2], ~\pm [0|3/2,0,-3/2]\right\}.\]
	  This means that $b\omega_1$ lies in  $W_{G_2}a\omega_1:=\{ax\omega_1|~x\in W_{G_2}\}$ if and only if $b =\pm a$. It follows from Theorem \ref{thm::blocks} that $\mc O_{a\omega_1} =\mc O_{-a\omega_1}$.

	  Suppose that $b\omega_1 \in \text{Irr}\mc O_{a\omega_1}.$ By Theorem \ref{thm::blocks},  there is $k\in \Z$ and $w\in W_{G_2}$ such that $$a\omega_1=w(b\omega_1+k(\delta+\vare_1))= bw\omega_1+kw\vare_1 +k\delta.$$
	  This implies that $k=0$. As a consequence, we have $a\omega_1 =bw\omega_1$ and so $b=\pm a$. This completes the proof.
\end{proof}

\subsection{Tilting characters}
For $\ell\in\mathbb{N}$ and $k\in\mathbb{Z}$, let ${}_{\ell}\lambda_k$ be the unique anti-dominant weight in $W_{G_2}([k~|~k,\frac{3\ell-k}{2},\frac{-3\ell-k}{2}])$. Let $Q_k$ be the set of minimal length representatives of the left cosets $Q^{(k)}$ in $W_{G_2}$, where $Q^{(k)}=\{w\in W_{G_2}~|~w({}_\ell\la_k)={}_\ell\la_k\}$. Let
${}_\ell\la_k^w=w({}_\ell\la_k)$ for any $w\in W_{G_2}$.

{Note $W_{G_2}=\langle s_{1},s_{2}\rangle$. Denote
\begin{align*}
  K_1=\{e,s_1,s_2s_1,s_1s_2s_1,s_2s_1s_2s_1,s_1s_2s_1s_2s_1\};\\
  K_2=\{e,s_2,s_1s_2,s_2s_1s_2,s_1s_2s_1s_2,s_2s_1s_2s_1s_2\}.
\end{align*}
In the following theorem, we use the Bruhat order ``$\leq$'' on $W_{G_2}$ by regarding $s_{1}$ and $s_{2}$ as simple reflections.}

\begin{thm} For any $\ell\in\mathbb{N}$, $k\in\mathbb{Z}$, the following formulas hold.
\begin{itemize}
\item[(1)] If $k=\pm3\ell\neq0$ or $k=0\neq\ell\neq1$, then $Q_k=K_1$. For any $w\in Q_k$, we have
    $$T_{_\ell\lambda_k^w}=\sum_{\sigma\in K_1,\sigma\leq w}M_{{}_\ell\la_k^{\sigma}}+\sum_{\sigma\leq ws_2}M_{{}_\ell\la_{k-1}^{\sigma}}.$$

\item[(2)] If $k=\pm\ell\neq0,1$, then $Q_k=K_2$. For any $w\in Q_k$, we have
    $$T_{_\ell\lambda_k^w}=\sum_{\sigma\in K_2,\sigma\leq w}M_{{}_\ell\la_k^{\sigma}}+\sum_{\sigma\leq ws_1}M_{{}_\ell\la_{k-1}^{\sigma}}.$$

\item[(3)] If $\ell=0$ and $k\neq0,1$, then $Q_k=K_2$. For any $w\in Q_k$, we have
    $$T_{_\ell\lambda_k^w}=\sum_{\sigma\in K_2,\sigma\leq w}(M_{{}_\ell\la_k^{\sigma}}+M_{{}_\ell\la_{k-1}^{\sigma}}).$$

\item[(4)] If $k=\pm3\ell+1$ with $\ell\neq0$, or, $k=1$ with $\ell\neq0,1$, then $Q_k=W_{G_2}$. For any $w\in K_1$, we have
\begin{align*}
    T_{{}_\ell\la_k^w}=\sum_{\sigma\leq w}M_{{}_\ell\la_k^{\sigma}}+\sum_{\sigma\in K_1,\sigma\leq w}M_{{}_\ell\la_{k-1}^{\sigma}}+\sum_{\sigma\leq w}M_{{}_\ell\la_{k-2}^{\sigma}};
  \end{align*}
while for any $w\in W_{G_2}\setminus K_1$, we have
\begin{align*}
    T_{{}_\ell\la_k^w}=\sum_{\sigma\leq w}M_{{}_\ell\la_k^{\sigma}}+\sum_{\sigma\in K_1,\sigma\leq w}M_{{}_\ell\la_{k-1}^{\sigma}}.
\end{align*}

\item[(5)] If $k=\pm\ell+1$ with $\ell\neq0$, then $Q_k=W_{G_2}$. For any $w\in K_2$, we have
\begin{align*}
    T_{{}_\ell\la_k^w}=\sum_{\sigma\leq w}M_{{}_\ell\la_k^{\sigma}}+\sum_{\sigma\in K_2,\sigma\leq w}M_{{}_\ell\la_{k-1}^{\sigma}}+\sum_{\sigma\leq w}M_{{}_\ell\la_{k-2}^{\sigma}};
  \end{align*}
while for any $w\in W_{G_2}\setminus K_2$, we have
\begin{align*}
    T_{{}_\ell\la_k^w}=\sum_{\sigma\leq w}M_{{}_\ell\la_k^{\sigma}}+\sum_{\sigma\in K_2,\sigma\leq w}M_{{}_\ell\la_{k-1}^{\sigma}}.
\end{align*}

\item[(6)] If $k=\ell=1$, then $Q_k=K_2$. For any $w\in Q_k$, we have
$$T_{_1\lambda_1^w}=\sum_{\sigma\in K_2,\sigma\leq w}M_{{}_1\la_1^{\sigma}}+\sum_{\sigma\in K_1,\sigma\leq ws_1}M_{{}_1\la_0^{\sigma}}.$$

\item[(7)] If $\ell=1$ and $k=0$, then $Q_k=K_1$. For any $w\in Q_k$, we have
    $$T_{_\ell\lambda_k^w}=\sum_{\sigma\in K_1,\sigma\leq w}M_{{}_\ell\la_k^{\sigma}}+\sum_{\sigma\in K_2,\sigma\leq ws_2}M_{{}_\ell\la_{k-1}^{\sigma}}.$$

\item[(8)] For $k=\ell=0$ we have $Q_k=\{e\}$ and
$$T_{_0\lambda_0}=M_{_0\lambda_0}+\sum_{\sigma\in K_2}M_{_0\lambda_{-1}^\sigma}.$$

\item[(9)] If $\ell=0$ and $k=1$ then $Q_k=K_2$ and we have
  \begin{align*}
    &T_{_0\la_1}=M_{_0\la_1}+M_{_0\la_0}+M_{_0\la_{-1}^{s_1s_2s_1s_2}}+M_{_0\la_{-1}^{s_2s_1s_2}}+M_{_0\la_{-1}^{s_1s_2}}+M_{_0\la_{-1}^{s_2}}+M_{_0\la_{-1}},\\
     &T_{_0\la_1^{s_2}}=M_{_0\la_1^{s_2}}+M_{_0\la_1}+M_{_0\la_0}+M_{_0\la_{-1}^{s_2s_1s_2}}+M_{_0\la_{-1}^{s_1s_2}}+M_{_0\la_{-1}^{s_2}}+M_{_0\la_{-1}},\\
      &T_{_0\la_1^{s_1s_2}}=M_{_0\la_1^{s_1s_2}}+M_{_0\la_{1}^{s_2}}+M_{_0\la_{1}}+M_{_0\la_0}+M_{_0\la_{-1}^{s_1s_2}}+M_{_0\la_{-1}^{s_2}}+M_{_0\la_{-1}},\\
       &T_{_0\la_1^{s_2s_1s_2}}=M_{_0\la_1^{s_2s_1s_2}}+M_{_0\la_1^{s_1s_2}}+M_{_0\la_1^{s_2}}+M_{_0\la_1}+M_{_0\la_0}+M_{_0\la_{-1}^{s_2}}+M_{_0\la_{-1}},\\
       &T_{_0\la_1^{s_1s_2s_1s_2}}=M_{_0\la_1^{s_1s_2s_1s_2}}+M_{_0\la_1^{s_2s_1s_2}}+M_{_0\la_1^{s_1s_2}}+M_{_0\la_1^{s_2}}+M_{_0\la_1}+M_{_0\la_0}+M_{_0\la_{-1}},\\
        &T_{_0\la_1^{s_2s_1s_2s_1s_2}}=M_{_0\la_1^{s_2s_1s_2s_1s_2}}+M_{_0\la_1^{s_1s_2s_1s_2}}+M_{_0\la_1^{s_2s_1s_2}}+M_{_0\la_1^{s_1s_2}}+M_{_0\la_1^{s_2}}+M_{_0\la_1}+M_{_0\la_0}.\\
  \end{align*}

  \item[(10)] If $k$ and $\ell$ do not satisfy any condition in (1)-(9), then $Q_k= W_{G_2}$. In this case, for any $w\in W_{G_2}$, we have
  \begin{equation*}
    T_{{}_\ell\la_k^w}=\sum_{\sigma\leq w}(M_{{}_\ell\la_k^{\sigma}}+M_{{}_\ell\la_{k-1}^{\sigma}}).
  \end{equation*}
  \end{itemize}
\end{thm}
\begin{proof}
All formulas are obtained by applying the translation functor, associated with tensoring with the adjoint module, to $T_{\lambda-2\delta}$, $T_{\lambda-(\delta-\epsilon_3)}$ or $T_{\lambda-(\delta+\epsilon_2)}$. Below we give more details to obtain the formulas in (9) as an example.

Applying the translation functor associated with tensoring with the adjoint module to $$T_{_0\la_1^{s_2s_1s_2s_1s_2}-(\delta+\epsilon_2)}=T_{[0|1,-\frac{1}{2},-\frac{1}{2}]}=
M_{[0|1,-\frac{1}{2},-\frac{1}{2}]}+M_{[0|-1,\frac{1}{2},\frac{1}{2}]}+M_{[0|\frac{1}{2},-1,\frac{1}{2}]}+M_{[0|-\frac{1}{2},-\frac{1}{2},1]},$$ we can obtain the characters of $T_{_0\la_1^{s_2s_1s_2s_1s_2}}$, $T_{_0\la_1^{s_2s_1s_2}}$ and $T_{_0\la_1^{s_2}}$. Here we make use of Propositions~\ref{prop:flags} and \ref{prop:comp:factor}. Indeed, Proposition \ref{prop:comp:factor} and the fact that any composition factor appearing in a $G_2$-Verma module is $1$ imply that all three tilting characters are multiplicity free.

Similarly, if we apply the translation functor to $$T_{_0\la_1^{s_1s_2s_1s_2}-(\delta-\epsilon_3)}=T_{[0|-1,\frac{1}{2},\frac{1}{2}]}=
M_{[0|-1,\frac{1}{2},\frac{1}{2}]}+M_{[0|-\frac{1}{2},-\frac{1}{2},1]}+M_{[0|\frac{1}{2},-1,\frac{1}{2}]},$$ then we get the characters of $T_{_0\la_1^{s_1s_2s_1s_2}}$, $T_{_0\la_1^{s_1s_2}}$ and $T_{_0\la_1}$.
\end{proof}

\appendix{}
\setcounter{secnumdepth}{1}
\section{Tilting characters of $\mathfrak{osp}(3|2)$}
\label{sect::app}

For the sake of completeness, in this appendix, we give the characters of the tilting modules of $\mf{osp}(3|2)$ in the BGG category $\mc O$.

For precise definition and more details of the contragredient simple Lie superalgebra $\mf{osp}(3|2)$ we refer the reader to the standard references, e.g., \cite{Kac77}. In this appendix we shall follow the notation used in \cite[Section 1.3.3]{CW12}. We shall use the simple system $\Pi=\{\delta-\vare,\vare\}$ with positive roots $\Phi^+=\{\delta\pm \vare, \delta, 2\delta,\vare\}$, so that the Weyl vector is
$\rho=-\frac{1}{2}\delta+\frac{1}{2}\epsilon$. Let $(,):\h^\ast \otimes \h^\ast \rightarrow \C$ denote the symmetric bilinear form determined by \[(\delta, \delta)=-1,\quad(\delta,\vare)=0,\quad(\vare, \vare)=1.\]
The Weyl group $W$ is isomorphic to $\Z_2\times\Z_2$.

A ($\rho$-shifted) weight $\lambda$ is typical if $(\la,\delta\pm\ep)\not=0$. It is atypical, if $(\la,\delta+\ep)=0$ or $(\la,\delta-\ep)=0$. In particular, if $\lambda=a\delta+b\epsilon$ with $a,b\in\C$, then it is atypical if and only if $a=\pm b$. Let $\mc O_\la$ denote the block in $\mc O$ containing $L_\la$ the simple module of highest weight $\la-\rho$. We shall use self-explanatory notations for Verma and tilting modules of that highest weight below. Similarly as in Section \ref{sect::preli}, we can define the integral Weyl group $W_\la$ and one can prove that $\text{Irr}\mc O_\la=\{w(\la+k\alpha)|(\la,\alpha)=0,k\in\Z,w\in W_\la\}$, which is an analogue of Theorem \ref{thm::blocks}.

First, we consider the case of $\mc O_\la$, when $\la$ is typical. If $\la$ is strongly typical, i.e., $(\la,\delta)\not=0$, then $\mc O_\la$ is equivalently to a block of underlying Lie algebra by \cite{Gor02b} again. If $\la$ typical, but not strongly typical, then an analogous argument as given in the proof of Theorem \ref{thm::tysty} shows that $\mc O_\la$ is equivalent to a block of $\mf{sl}(2)$. Therefore, the characters of tilting modules in typical blocks are completely determined. Thus, we can restrict ourselves to studying atypical blocks below.

\begin{thm}\label{ch:osp32}
\begin{itemize}
  \item[(1)] If $\lambda=a\delta+a\epsilon$ (respectively $a\delta-a\epsilon$), $(a\not\in\frac{1}{2}\mathbb{Z})$, then $W_\lambda=\{e\}$, $\mathrm{Irr} \mathcal{O}_\lambda=\{\lambda+k\alpha~|~k\in\mathbb{Z}\}$ and $T_\lambda=M_\lambda+M_{\lambda-\alpha}$, where $\alpha=\delta+\epsilon$ (respectively $\delta-\epsilon$).
  \item[(2)] If $\lambda=a\delta\pm a\epsilon$, $(a\in\mathbb{Z})$, then $W_\lambda=\langle s_\epsilon\rangle\cong\mathbb{Z}_2$ and $\mathrm{Irr} \mathcal{O}_\lambda=\{k(\delta\pm\epsilon)~|~k\in\mathbb{Z}\}$. For any $j\in\mathbb{Z}_{>0}$,
      \begin{align*}
        T_0=&M_0+M_{-\delta-\epsilon}+M_{-\delta+\epsilon};\\
        T_{\delta+\epsilon}=&M_{\delta+\epsilon}+M_{\delta-\epsilon}+M_{-\delta+\epsilon}+M_{-\delta-\epsilon}
        +M_{0};\\
        T_{j\delta+j\epsilon}=&M_{j\delta+j\epsilon}+M_{j\delta-j\epsilon}+M_{-j\delta+j\epsilon}+M_{-j\delta-j\epsilon}\\
        &+M_{(j-1)\delta+(j-1)\epsilon}+M_{(j-1)\delta-(j-1)\epsilon}+M_{-(j-1)\delta+(j-1)\epsilon}+M_{-(j-1)\delta-(j-1)\epsilon},\quad(j\neq1);\\
        T_{\delta-\epsilon}=&M_{\delta-\epsilon}+M_{-\delta-\epsilon}+M_{0};\\
        T_{j\delta-j\epsilon}=&M_{j\delta-j\epsilon}+M_{-j\delta-j\epsilon}+M_{(j-1)\delta-(j-1)\epsilon}+M_{-(j-1)\delta-(j-1)\epsilon},\quad(j\neq1);\\
        T_{-j\delta+j\epsilon}=&M_{-j\delta+j\epsilon}+M_{-j\delta-j\epsilon}+M_{-(j+1)\delta+(j+1)\epsilon}+M_{-(j+1)\delta-(j+1)\epsilon};\\
      T_{-j\delta-j\epsilon}=&M_{-j\delta-j\epsilon}+M_{-(j+1)\delta-(j+1)\epsilon}.
      \end{align*}
  \item[(3)] If $\lambda=a\delta\pm a\epsilon$, $(a\in\hf+\mathbb{Z})$, then $W_\lambda=\langle s_\delta,s_\epsilon\rangle\cong\mathbb{Z}_2\times\mathbb{Z}_2$ and $\mathrm{Irr} \mathcal{O}_\lambda=\{(k+\frac{1}{2})(\delta\pm\epsilon)~|~k\in\mathbb{Z}\}$. For any $j\in\hf+\mathbb{Z}_{>0}$,
      \begin{align*}
       T_{\frac{1}{2}\delta+\frac{1}{2}\epsilon}=&M_{\frac{1}{2}\delta+\frac{1}{2}\epsilon}+M_{\frac{1}{2}\delta-\frac{1}{2}\epsilon}+
       M_{-\frac{1}{2}\delta-\frac{1}{2}\epsilon}+M_{-\frac{1}{2}\delta+\frac{1}{2}\epsilon};\\
       T_{\frac{1}{2}\delta-\frac{1}{2}\epsilon}=&M_{\frac{1}{2}\delta-\frac{1}{2}\epsilon}+M_{-\frac{1}{2}\delta+\frac{1}{2}\epsilon}+
       M_{-\frac{1}{2}\delta-\frac{1}{2}\epsilon}+M_{-\frac{3}{2}\delta-\frac{3}{2}\epsilon};\\
       T_{-\frac{1}{2}\delta+\frac{1}{2}\epsilon}=&M_{-\frac{1}{2}\delta+\frac{1}{2}\epsilon}+M_{-\frac{1}{2}\delta-\frac{1}{2}\epsilon}+M_{-\frac{3}{2}\delta+\frac{3}{2}\epsilon}+M_{-\frac{3}{2}\delta-\frac{3}{2}\epsilon};\\
       T_{-\frac{1}{2}\delta-\frac{1}{2}\epsilon}=&M_{-\frac{1}{2}\delta-\frac{1}{2}\epsilon}+M_{-\frac{3}{2}\delta-\frac{3}{2}\epsilon};\\
       T_{j\delta+j\epsilon}=&M_{j\delta+j\epsilon}+M_{j\delta-j\epsilon}+M_{(j-1)\delta+(j-1)\epsilon}+M_{(j-1)\delta-(j-1)\epsilon};\\
       T_{-j\delta+j\epsilon}=&M_{-j\delta+j\epsilon}+M_{-j\delta-j\epsilon}+M_{-(j+1)\delta+(j+1)\epsilon}+M_{-(j+1)\delta-(j+1)\epsilon};\\
      T_{j\delta-j\epsilon}=&M_{j\delta-j\epsilon}+M_{(j-1)\delta-(j-1)\epsilon};\\
      T_{-j\delta-j\epsilon}=&M_{-j\delta-j\epsilon}+M_{-(j+1)\delta-(j+1)\epsilon}.
      \end{align*}
\end{itemize}
\end{thm}
\begin{proof}
When $\la\not=\frac{1}{2}\delta\pm\frac{1}{2}\epsilon$, all tilting characters $T_\lambda$ can be obtained by applying the translation functor, associated with tensoring with the standard module $L_{\delta+\rho}$, on $T_{\lambda-\delta}=M_{\lambda-\delta}$.

In the case of $\lambda=\frac{1}{2}\delta\pm\frac{1}{2}\epsilon$, we obtain their tilting characters by applying the translation functor associated with $L_{\delta+\rho}$ to $T_{-\frac{1}{2}\delta+\frac{1}{2}\epsilon}=M_{-\frac{1}{2}\delta+\frac{1}{2}\epsilon}+M_{-\frac{1}{2}\delta-\frac{1}{2}\epsilon}+M_{-\frac{3}{2}\delta+\frac{3}{2}\epsilon}+M_{-\frac{3}{2}\delta-\frac{3}{2}\epsilon}$ and to $T_{-\frac{1}{2}\delta-\frac{1}{2}\epsilon}=M_{-\frac{1}{2}\delta-\frac{1}{2}\epsilon}+M_{-\frac{3}{2}\delta-\frac{3}{2}\epsilon}$, respectively.

{All calculations above are straightforward using Lemma \ref{prop:flags} except in the case of $T_{\hf\delta+\hf\epsilon}$. Indeed, here we need to use the fact that $[T_{\frac{1}{2}\delta+\frac{1}{2}\epsilon}:M_{-\frac{1}{2}\delta+\frac{1}{2}\epsilon}]=1$, which is equivalent to $[M_{\frac{1}{2}\delta-\frac{1}{2}\epsilon}:L_{-\frac{1}{2}\delta-\frac{1}{2}\epsilon}]=1$. This can be seen as follows.
First, we note that the $\rho$-shifted weight space of weight ${-\frac{1}{2}\delta-\frac{1}{2}\epsilon}$ in the Verma module $M_{\frac{1}{2}\delta-\frac{1}{2}\epsilon}$  is two-dimensional. Now, the vector $E_{\delta}E_{-\delta}v_{\delta-\epsilon}$ is a nonzero multiple of the highest weight vector $v_{\delta-\epsilon}$ in $M_{\hf\delta-\hf\epsilon}$. From this, we conclude that the dimension of the space of primitive vectors of $\rho$-shifted weight $-\hf\delta-\hf\epsilon$ in $M_{\frac{1}{2}\delta-\frac{1}{2}\epsilon}$ cannot exceed $1$, and hence we have $[M_{\frac{1}{2}\delta-\frac{1}{2}\epsilon}:L_{-\frac{1}{2}\delta-\frac{1}{2}\epsilon}]= 1$, as claimed.}
\end{proof}

\begin{rem}
Note that in Theorem \ref{ch:osp32}(3) all weights are integral. Thus, the results of Bao and Wang in \cite{BW18} applies. Indeed, the formulas in these cases have been computed earlier by them. We are grateful to them for communicating these formulas.
\end{rem}


\end{document}